\pdfoutput=1
\documentclass[letterpaper,    
               english,        
               12pt,           
               twoside]        
               {amsart}        
\usepackage[english]{babel}                   
\usepackage[bookmarks,                        
            bookmarksopen=true,               
            bookmarksnumbered=true,           
            pdfauthor={Christof Puhle},       
            pdftitle={Riemannian manifolds    
                      with structure group PSU(3)},
            pdfkeywords={PSU(3)-structures,
                         connections with torsion,
                         holonomy},
            colorlinks,                       
            linkcolor=black,                  
            citecolor=black,                  
            urlcolor=black,                   
            plainpages=false]                 
            {hyperref}                        
\usepackage[nobysame]
           {amsrefs}                          
\usepackage{fourier}                          
\usepackage{bm}                               
\usepackage{color}                            
\advance\oddsidemargin by -0.93cm
\advance\evensidemargin by -0.92cm
\advance\topmargin by -0.55cm
\theoremstyle{plain}
\newtheorem{cor}{Corollary}[section]
\newtheorem{lem}{Lemma}[section]
\newtheorem{thm}{Theorem}[section]            
\newtheorem{prop}{Proposition}[section]
\theoremstyle{definition}
\newtheorem{exa}{Example}[section]

\newcommand{\beqn}[1][*]{\begin{equation#1}}
\newcommand{\eeqn}[1][*]{\end{equation#1}}
\newcommand{\btab}[1]{\renewcommand{\arraystretch}{1.2}\begin{center}\begin{tabular}{#1}}
\newcommand{\etab}{\end{tabular}\end{center}\renewcommand{\arraystretch}{1.0}}
\newcommand{\bite}{\begin{itemize}}
\newcommand{\eite}{\end{itemize}}
\newcommand{\benu}{\begin{enumerate}}
\newcommand{\eenu}{\end{enumerate}}
\newcommand{\bcen}{\begin{center}}
\newcommand{\ecen}{\end{center}}

\newcommand{\dsty}{\displaystyle}
\newcommand{\mca}[1]{\mathcal{#1}}
\newcommand{\mrm}[1]{\mathrm{#1}}
\newcommand{\mfr}[1]{\mathfrak{#1}}
\def\side#1{\ifvmode\leavevmode\fi\vadjust{\vbox to0pt{\vss
\hbox to 0pt{\hskip\hsize\hskip1em                     
\vbox{\hsize2cm\small\raggedright\pretolerance10000       
\noindent\textcolor{red}{#1}\hfill}\hss}\vbox to8pt{\vfil}\vss}}}
\newcommand{\x}{\ensuremath{\times}}
\newcommand{\op}{\ensuremath{\oplus}}
\newcommand{\ox}{\ensuremath{\otimes}}

\newcommand{\lan}{\ensuremath{\left\langle}}
\newcommand{\ran}{\ensuremath{\right\rangle}}
\newcommand{\hook}{\ensuremath{\lrcorner\,}}
\newcommand{\ra}{\ensuremath{\rightarrow}}

\newcommand{\setsep}{\ensuremath{\,\big|\,}}
\newcommand{\ex}{\ensuremath{\exists\,}}
\newcommand{\all}{\ensuremath{\forall\,}}

\newcommand{\sdots}{\ensuremath{{\scriptstyle\ldots}}}
\newcommand{\R}{\ensuremath{\mathbb{R}}}

\newcommand{\Lie}{\ensuremath{\mathcal{L}}}

\newcommand{\W}{\ensuremath{\mathcal{W}}}

\newcommand{\T}{\ensuremath{\mathcal{T}}}
\newcommand{\A}{\ensuremath{\mathcal{A}}}
\newcommand{\G}{\ensuremath{\mathrm{G}}}

\newcommand{\vphi}{\ensuremath{\varphi}}    
\newcommand{\vrho}{\ensuremath{\varrho}}

\newcommand{\Ker}{\ensuremath{\mathrm{Ker}}}

\newcommand{\diag}{\ensuremath{\mathrm{diag}}}

\newcommand{\pr}{\ensuremath{\mathrm{pr}}}
\newcommand{\Ric}{\ensuremath{\mathrm{Ric}}}
\newcommand{\Scal}{\ensuremath{\mathrm{Scal}}}
\newcommand{\hol}{\ensuremath{\mathfrak{hol}}}
\newcommand{\Hol}{\ensuremath{\mathrm{Hol}}}

\newcommand{\GL}{\ensuremath{\mathrm{GL}}}

\newcommand{\SL}{\ensuremath{\mathrm{SL}}}
\newcommand{\PSL}{\ensuremath{\mathrm{PSL}}}

\newcommand{\su}{\ensuremath{\mathfrak{su}}}
\newcommand{\SU}{\ensuremath{\mathrm{SU}}}
\newcommand{\PSU}{\ensuremath{\mathrm{PSU}}}
\newcommand{\psu}{\ensuremath{\mathfrak{psu}}}

\newcommand{\Orth}{\ensuremath{\mathrm{O}}}
\newcommand{\so}{\ensuremath{\mathfrak{so}}}
\newcommand{\SO}{\ensuremath{\mathrm{SO}}}
\newcommand{\spin}{\ensuremath{\mathfrak{spin}}}
\newcommand{\Spin}{\ensuremath{\mathrm{Spin}}}
\newcommand{\g}{\ensuremath{\mathfrak{g}}}
\newcommand{\h}{\ensuremath{\mathfrak{h}}}
\newcommand{\m}{\ensuremath{\mathfrak{m}}}
\newcommand{\zero}{\ensuremath{\mathfrak{0}}}
\newcommand{\tor}{\ensuremath{\mathfrak{t}}}
\begin{document}
\thispagestyle{empty}
\date{\today}
\title{Riemannian manifolds with structure group PSU(3)}
\author{Christof Puhle}
\address{
Institut f\"ur Mathematik \newline\indent
Humboldt-Universit\"at zu Berlin\newline\indent
Unter den Linden 6\newline\indent
10099 Berlin, Germany}
\email{\noindent puhle@math.hu-berlin.de}
\urladdr{www.math.hu-berlin.de/~puhle}
\subjclass[2000]{Primary 53C25; Secondary 81T30}
\keywords{$\PSU(3)$-structures, connections with torsion, holonomy}
\begin{abstract}
We study $8$-dimensional Riemannian manifolds that admit a $\PSU(3)$-structure. We classify these structures by their intrinsic torsion and characterize the corresponding classes via differential equations. Moreover, we consider a connection defined by a $3$- and a $4$-form that preserves the underlying structure. Finally, we discuss the geometry of these manifolds relatively to the holonomy algebra of this connection.
\end{abstract}
\maketitle
\setcounter{tocdepth}{1}
\bcen
\begin{minipage}{0.7\linewidth}
    \begin{small}
      \tableofcontents
    \end{small}
\end{minipage}
\ecen
\pagestyle{headings}
%
%
%
%
\section{Introduction}\noindent
In the first half of the last century several mathematicians investigated the action of the general linear group on the space of alternating trilinear forms (see \cites{Rei07,Sch31,Gur35}). Their results can be summerized as follows (cf. \cite{Hit01}): If a $3$-form $\rho\in\Lambda^3{\R^n}^{\ast}$ lies in an open orbit under the natural action of $\GL\left(n\right)$, then $n=6,7,8$ and the stabilizer of $\rho$ in $\GL\left(n\right)$ is a real form of one of the complex groups
\beqn
\SL(3)\x\SL(3),\quad \G_2, \quad \PSL(3)
\eeqn
respectively. Therefore, if an oriented Riemannian manifold $\left(M^n,g\right)$ admits a global $3$-form $\rho$ that lies in an open orbit, its dimension is $n=6,7,8$ and there exists a $\G$-structure on $\left(M^n,g\right)$ with $\G$ one of the compact groups
\beqn
\SU(3),\quad \G_2, \quad \PSU(3)
\eeqn
respectively. In contrast to the first two cases, the latter is, apart from the study in \cites{Hit01,Wit05,Wit08,Nur08}, largely unexplored so far.

One particular $3$-form lying in an open $\GL\left(8\right)$-orbit is familiar to most theoretical physicists. Starting with Gell-Mann's Eightfold Way \cite{GN64} -- a first classification scheme for elementary particles which led to the quark model -- an important algebraic role in particle physics is played by $f\in\Lambda^3{\R^8}^{\ast}$ whose non-zero components are
\beqn
f_{123}=1,\quad f_{147}=-f_{156}=f_{246}=f_{257}=f_{345}=-f_{367}=\frac{1}{2},\quad f_{458}=f_{678}=\frac{\sqrt{3}}{2}.
\eeqn
Up to a factor, these components are the structure constants,
\beqn
\left[\lambda_i,\lambda_j\right]=2\,\mrm{i}\sum_kf_{ijk}\lambda_k,
\eeqn
of $\SU(3)$ with respect to the Gell-Mann matrices $\lambda_i$, a generalization of the Pauli matrices which is used to represent the $8$ types of gluons that mediate the so-called strong force in quantum chromodynamics. From the above point of view, $f$ lies in an open orbit under the action of $\GL(8)$ and its stabilizer is $\PSU(3)$ (see \cite{Wit05}).

Firstly, we study general aspects of $\PSU(3)$-structures $\left(M^8,g,\rho\right)$. To begin with, we follow the method of \cite{Fri03} and classify these structures with respect to the algebraic type of the corresponding intrinsic torsion tensor $\Gamma$. There are six irreducible $\PSU(3)$-modules $\W_1,\ldots,\W_6$ in the decomposition of the space of possible intrinsic torsion tensors (see theorem \ref{thm:1}),
\beqn
\Gamma\in\W_1\op\ldots\op\W_6.
\eeqn
Thus, there exist $64$ classes in this scheme. Up to \cite{Nur08}, the only $\PSU(3)$-structures studied were those with $\Gamma\in\W_6$ . Indeed, in our characterization of the $64$ classes by differential equations in terms of $\rho$ (see theorem \ref{thm:3}) we conclude that the case $\Gamma\in\W_6$ is equivalent to
\beqn
d\rho=0,\quad \delta\rho=0,
\eeqn
the general assumption in the considerations of \cites{Hit01,Wit05,Wit08}. The so-called restricted nearly integrable $\SU(3)$ structures of \cite{Nur08} correspond to the case $\Gamma\in\W_1\op\W_2\op\W_3$. Via $\PSU(3)$-invariant isomorphisms, we identify the projection of $\Gamma$ onto $\W_1\op\ldots\op\W_5$ with a pair $\left(T^c,F^c\right)$ of differential forms on $\left(M^8,g,\rho\right)$, a $3$-form $T^c$ and a $4$-form $F^c$. Restricting the considerations to $\PSU(3)$-structures whose intrinsic torsion is of type $\W_1\op\ldots\op\W_5$, there exists a uniquely determined metric connection $\nabla^c$ that preserves the underlying structure (cf.\ theorem \ref{thm:2}),
\beqn
g\left(\nabla^c_XY,Z\right)=g\left(\nabla^{g}_XY,Z\right)+\frac{1}{2}\, T^c\left(X,Y,Z\right)+\left(\left(X\hook\rho\right)\hook F^c\right)\left(Y,Z\right).
\eeqn
The torsion tensor $\T^c$ of this so-called characteristic connection realizes different algebraic types (see corollary \ref{cor:2}). For example, $\T^c$ is totally skew-symmetric if and only if $F^c=0$, that is if $\Gamma\in\W_1\op\W_2\op\W_3$, $\T^c$ is cyclic if and only if $T^c=0$, or equivalently if $\Gamma\in\W_4\op\W_5$. The $\nabla^c$-parallelism of $\T^c$ is equivalent (see proposition \ref{prop:1}) to
\beqn
\nabla^cT^c=0,\quad\nabla^cF^c=0.
\eeqn
Moreover, we compute necessary conditions for $T^c$, $F^c$ and $\mca{R}^c$, the curvature operator of $\nabla^c$, in the case of $\nabla^c\T^c=0$ (see corollaries \ref{cor:3} and \ref{cor:1}). 

Secondly, our aim is the construction and classification of $\PSU(3)$-manifolds with $\nabla^c$-parallel $\T^c\neq0$. These are Ambrose-Singer manifolds, i.e.\
\beqn
\nabla^c\T^c=0,\quad\nabla^c\mca{R}^c=0,
\eeqn
and the holonomy algebra $\hol\left(\nabla^c\right)$ of the characteristic connection is one of six subalgebrae of $\psu(3)$ (see lemma \ref{lem:8} and propositions \ref{prop:6} and \ref{prop:4}). We then assume
\beqn
\dim\left(\hol\left(\nabla^c\right)\right)>1.
\eeqn
For each of the remaining four holonomy algebrae we describe the admissible tensors $\T^c$, $\mca{R}^c$ and discuss the respective geometry. Summerizing these considerations, the main classification result is that one of the following holds:
\begin{itemize}
\item[i)] $\left(M^8,g,\rho\right)$ is of type $\W_1\op\W_2\op\W_3$ (see theorem \ref{thm:5}) and admits a $\Spin(7)$-structure preserved by $\nabla^c$ (see theorem \ref{thm:4}). Moreover, if $\left(M^8,g,\rho\right)$ is regular, it is a principal $\mrm{S}^1$-bundle and a Riemannian submersion over a cocalibrated $\G_2$-manifold of certain type (cf.\ theorems \ref{thm:6}, \ref{thm:7}, \ref{thm:8} and \ref{thm:11}).
\item[ii)] $\left(M^8,g,\rho\right)$ is either of type $\W_3$ and $\Scal^g>0$ or of type $\W_5$ and $\Scal^g<0$ (see theorem \ref{thm:9}). In each of these two cases, $M^8$ is locally isomorphic to a unique homogeneous space with isotropy group $\SO(3)$ (cf.\ theorem \ref{thm:10}).
\end{itemize}
Conversely, we can reconstruct the $\PSU(3)$-structure in many of these situations (see theorems \ref{thm:4}, \ref{thm:6} and \ref{thm:10}, together with example \ref{exa:1}).

The paper is structured as follows: In \autoref{sec:2} we study the algebra related to the action of the group $\PSU(3)$. We then classify $\PSU(3)$-structures $\left(M^8,g,\rho\right)$ with respect to their intrinsic torsion in \autoref{sec:3}. Section \ref{sec:4} is devoted to the characterization of these classes by differential equations in terms of $\rho$. We develop the notion of characteristic connection and study its torsion and curvature in \autoref{sec:5}. In the last section we discuss the geometry of $M^8$ relatively to the holonomy algebra of this connection.
%
%
%
%
\section{The group \texorpdfstring{$\PSU\left(3\right)$}{PSU(3)}}\label{sec:2}\noindent
We first introduce some notation. $\R^8$ denotes the $8$-dimensional Euclidian space. We fix an orientation in $\R^8$ and use its scalar product to identify $\R^8$ with its dual space ${\R^8}^{\ast}$. Let $\left(e_1,\ldots,e_8\right)$ denote an oriented orthonormal basis and $\Lambda^k$ the space of $k$-forms of $\R^8$. The family of operators
\beqn
\sigma_j:\Lambda^k\x\Lambda^l\ra\Lambda^{k+l-2j},
\eeqn
\begin{align*}
\sigma_{j}\left(\alpha,\beta\right)&:=\sum_{i_1<\ldots<i_j}\left(e_{i_1}\hook\ldots\hook e_{i_j}\hook\alpha\right)\wedge\left(e_{i_1}\hook\ldots\hook e_{i_j}\hook\beta\right), &
\sigma_0\left(\alpha,\beta\right)&:=\alpha\wedge\beta
\end{align*}
allows us to define an inner product and a norm on $\Lambda^k$ as
\begin{align*}
\lan\alpha,\beta\ran&:=\sigma_k\left(\alpha,\beta\right), &
\|\alpha\|&:=\sqrt{\sigma_k\left(\alpha,\alpha\right)}.
\end{align*}
The special orthogonal group $\SO(8)$ acts on $\Lambda^k$ via the adjoint representation $\vrho$. The differential
\beqn
\vrho_* : \mfr{so}(8) \rightarrow \mfr{so}\left(\Lambda^k\right)
\eeqn
of this faithful representation can be expressed as
\beqn
\vrho_\ast\left(\omega\right)\left(\alpha\right) = \sigma_1\left(\omega,\alpha\right)
\eeqn
by identifying the Lie algebra $\so(8)$ with the space of $2$-forms $\Lambda^2$.

The $8$-dimensional, compact, connected Lie group $\PSU(3)\subset\SO(8)$ can be described as the isotropy group of the $3$-form
\beqn\label{eqn:1}
\rho = e_{246}-e_{235}-e_{145}-e_{136}+\left(e_{12}+e_{34}-2\,e_{56}\right)\wedge e_7+\sqrt{3}\left(e_{12}-e_{34}\right)\wedge e_8. \tag{$\star$}
\eeqn
We use the notation $e_{i_1\ldots i_j}$ for the exterior product
\beqn
e_{i_1}\wedge\ldots\wedge e_{i_j}.
\eeqn

There are three $\PSU(3)$-equivariant operators defined on $\Lambda^k$, namely the Hodge operator $\ast:\Lambda^k\ra\Lambda^{8-k}$, $\sigma_+:\Lambda^k\ra\Lambda^{k+1}$ and $\sigma_-:\Lambda^k\ra\Lambda^{k-1}$, the latter two defined by
\begin{align*}
\sigma_+\left(\alpha\right)&:=\sigma_1\left(\rho,\alpha\right), & \sigma_-\left(\alpha\right)&:=\sigma_2\left(\rho,\alpha\right).
\end{align*}
Using these, we decompose $\Lambda^k$ into irreducible $\PSU(3)$-modules of real type. First we restate the results of \cite{Wit05} regarding $\Lambda^1$, $\Lambda^2$, $\Lambda^3$. The space $\Lambda^1$ is an irreducible $\PSU(3)$-module and
\beqn
\Lambda^2=\Lambda^2_8\op\Lambda^2_{20}
\eeqn
splits into the two irreducible $\PSU(3)$-modules
\begin{align*}
\Lambda^2_8 &:= \left\{\sigma_+\left(\xi\right)\setsep\xi\in\Lambda^1\right\}, &
\Lambda^2_{20} &:= \left\{\omega\in\Lambda^2\setsep\omega\wedge\ast\rho=0\right\}.
\end{align*}
The space of $3$-forms
\beqn
\Lambda^3=\Lambda^3_1\op\Lambda^3_8\op\Lambda^3_{20}\op\Lambda^3_{27}
\eeqn
decomposes into four irreducible $\PSU(3)$-modules:
\begin{align*}
\Lambda^3_1 &:= \left\{t\cdot\rho\setsep t\in\R\right\}, \\
\Lambda^3_8 &:= \left\{\ast\left(\omega\wedge\rho\right)\setsep\omega\in\Lambda^2_8\right\},\\
\Lambda^3_{20} &:= \left\{\sigma_+\left(\omega\right)\setsep\omega\in\Lambda^2_{20}\right\}, \\
\Lambda^3_{27} &:= \left\{T\in\Lambda^3\setsep T\wedge\rho=0\ \text{and}\ T\wedge\ast\rho=0\right\}.
\end{align*}
The subscript indicates the dimension of a module. We then define
\begin{align*}
\Lambda^4_{8} &:= \left\{\sigma_+\left(T\right)\setsep T\in\Lambda^3_8\right\}, &
\Lambda^4_{27} &:= \left\{\sigma_+\left(T\right)\setsep T\in\Lambda^3_{27}\right\}.
\end{align*}
\begin{lem}\label{lem:3}
The operators $\sigma_+$, $\sigma_-$ satisfy
\bite
\item[i)] $\dsty\sigma_-\circ\sigma_+\left(\xi\right)=6\,\xi\quad\all\xi\in\Lambda^1$,
\item[ii)] $\dsty\sigma_+\left(\omega\right)=0\quad\all\omega\in\Lambda^2_8$,
\item[iii)] $\dsty\sigma_-\circ\sigma_+\left(T\right)=6\,T\quad\all T\in\Lambda^3_8$,
\item[iv)] $\dsty\sigma_+\circ\sigma_-\left(T\right)=12\,T\quad\all T\in\Lambda^3_{20}$,
\item[v)] $\dsty\sigma_-\circ\sigma_+\left(T\right)=16\,T\quad\all T\in\Lambda^3_{27}$,
\item[vi)] $\dsty\sigma_-\left(F\right)=0\quad\all F\in\ast\Lambda^4_8\op\ast\Lambda^4_{27}$.
\eite
\end{lem}\noindent
Therefore, $\sigma_+$ is a $\PSU(3)$-equivariant isomorphism between $\Lambda^3_i$ and $\Lambda^4_i$ for $i=8,27$.
\begin{prop}
The decomposition
\beqn
\Lambda^4=\Lambda^4_{8}\op\Lambda^4_{27}\op\ast\Lambda^4_{8}\op\ast\Lambda^4_{27}
\eeqn
splits the space of $4$-forms of $\R^8$ into irreducible $\PSU(3)$-modules.
\end{prop}\noindent
A direct computation yields
\begin{lem}\label{lem:5}
Let $T\in\Lambda^3$ be a non-zero $3$-form. The equation $\rho\wedge T=0$ holds if and only if $T\in\Lambda^3_1\op\Lambda^3_{27}$. Moreover,
\beqn
\rho\hook\left(\rho\wedge T\right)=10\,T\quad\all T\in\Lambda^3_8.
\eeqn
\end{lem}
\begin{lem}\label{lem:4}
Let $F\in\Lambda^4$ be a non-zero $4$-form. The equation $\rho\hook F=0$ holds if and only if $F\in\Lambda^4_8\op\Lambda^4_{27}\op\ast\Lambda^4_{27}$. Moreover,
\beqn
\rho\wedge\left(\rho\hook F\right)=10\,F\quad\all F\in\ast\Lambda^4_8.
\eeqn
\end{lem}
The Lie algebra $\so\left(8\right)$ splits into $\psu\left(3\right)=\Lambda^2_8$ spanned by the forms $\omega_i:=e_i\hook\rho$,
\begin{align*}
\omega_1&=-e_{36}-e_{45}+e_{27}+\sqrt{3}\,e_{28},&
\omega_2&=e_{46}-e_{35}-e_{17}-\sqrt{3}\,e_{18},\\
\omega_3&=e_{16}+e_{25}+e_{47}-\sqrt{3}\,e_{48},&
\omega_4&=-e_{26}+e_{15}-e_{37}+\sqrt{3}\,e_{38},\\
\omega_5&=-e_{14}-e_{23}-2\,e_{67},&
\omega_6&=-e_{13}+e_{24}+2\,e_{57},\\
\omega_7&=e_{12}+e_{34}-2\,e_{56},&
\omega_8&=\sqrt{3}\,e_{12}-\sqrt{3}\,e_{34},
\end{align*}
and its orthogonal complement $\m=\Lambda^2_{20}$. We now decompose $\R^8\ox\m$ into irreducible $\PSU(3)$-modules. Let us define the $\PSU(3)$-equivariant linear maps
\begin{align*}
\Phi_1 &:\R^8\ox\m\ra\Lambda^3, &
\Phi_2 &:\R^8\ox\m\ra\Lambda^4,\\
\Theta_1 &:\Lambda^3\ra\R^8\ox\m, &
\Theta_2 &:\Lambda^4\ra\R^8\ox\m
\end{align*}
by assigning
\begin{align*}
&\Phi_1\left(X\ox\omega\right):=X\wedge\omega, & 
&\Phi_2\left(X\ox\omega\right):=\left(X\hook\rho\right)\wedge\omega,\\
&\Theta_1\left(T\right):=-\frac{1}{2}\sum_ie_i\ox\pr_\m\left(e_i\hook T\right), &
&\Theta_2\left(F\right):=-\sum_ie_i\ox\pr_\m\left(\left(e_i\hook\rho\right)\hook F\right),
\end{align*}
where $\pr_\m$ denotes the projection onto $\m$.
\begin{lem}\label{lem:6}
If $X\in\R^8$ and $F\in\ast\Lambda^4_8\op\ast\Lambda^4_{27}$, then
\beqn
\left(X\hook\rho\right)\hook F\in\m.
\eeqn
\end{lem}
\begin{lem}
$\Lambda^3_1$ is not contained in the image of $\Phi_1$.
\end{lem}
\begin{proof}
Suppose there exist $X\in\R^8$ and $\omega\in\m$ such that $\rho=X\wedge\omega$. It follows that
\beqn
\rho\wedge\ast\rho=\left(X\wedge\omega\right)\wedge\ast\rho=X\wedge\left(\omega\wedge\ast\rho\right)=0,
\eeqn
a contradiction.
\end{proof}
\begin{lem}\label{lem:1}
The image of $\Phi_1$ is $\Lambda^3_8\op\Lambda^3_{20}\op\Lambda^3_{27}$ and the maps $\Phi_1\circ\Theta_1$, $\Phi_1\circ\Theta_2$ satisfy
\bite
\item[i)] $\dsty\Phi_1\circ\Theta_1\left(T\right)=-\tfrac{1}{2}T\quad\all T\in\Lambda^3_8$,
\item[ii)] $\dsty\Phi_1\circ\Theta_1\left(T\right)=-T\quad\all T\in\Lambda^3_{20}$,
\item[iii)] $\dsty\Phi_1\circ\Theta_1\left(T\right)=-\tfrac{4}{3}T\quad\all T\in\Lambda^3_{27}$,
\item[iv)] $\dsty\Phi_1\circ\Theta_2\left(F\right)=0\quad\all F\in\ast\Lambda^4_8\op\ast\Lambda^4_{27}$.
\eite
\end{lem}
\begin{lem}\label{lem:2}
$\Phi_2$ is surjective and the maps $\Phi_2\circ\Theta_1$, $\Phi_2\circ\Theta_2$ satisfy
\bite
\item[i)] $\dsty\Phi_2\circ\Theta_1\left(T\right)=\tfrac{1}{2}\,\sigma_+\left(T\right)\quad\all T\in\Lambda^3_8$,
\item[ii)] $\dsty\Phi_2\circ\Theta_1\left(T\right)=-\tfrac{1}{3}\,\sigma_+\left(T\right)\quad\all T\in\Lambda^3_{27}$,
\item[iii)] $\dsty\Phi_2\circ\Theta_2\left(F\right)=-18\,F\quad\all F\in\ast\Lambda^4_8$,
\item[iv)] $\dsty\Phi_2\circ\Theta_2\left(F\right)=-8\,F\quad\all F\in\ast\Lambda^4_{27}$.
\eite
\end{lem}\noindent
Let us introduce the following subspaces of $\R^8\ox\m$:
\begin{align*}
\W_1&:=\Theta_1\left(\Lambda^3_8\right), &
\W_2&:=\Theta_1\left(\Lambda^3_{20}\right), &
\W_3&:=\Theta_1\left(\Lambda^3_{27}\right), \\
\W_4&:=\Theta_2\left(\ast\Lambda^4_8\right), &
\W_5&:=\Theta_2\left(\ast\Lambda^4_{27}\right), &
\W_6&:=\Ker\left(\Phi_1\right)\cap\Ker\left(\Phi_2\right).
\end{align*}
Since there exists only one $70$-dimensional irreducible $\PSU(3)$-module in the decomposition of $\R^8\ox\m$, we deduce the following using lemmata \ref{lem:1} and \ref{lem:2}:
\begin{thm}\label{thm:1}
The space $\R^8\ox\m$ splits into six irreducible $\PSU(3)$-modules:
\beqn
\R^8\ox\m=\W_1\op\W_2\op\W_3\op\W_4\op\W_5\op\W_6.
\eeqn
\end{thm}
%
%
%
%
\section{The sixty-four classes of \texorpdfstring{$\PSU\left(3\right)$}{PSU(3)}-structures}\label{sec:3}\noindent
We define a \emph{$\PSU(3)$-structure/manifold} as a triple $\left(M^8,g,\rho\right)$ consisting of a Riemannian $8$-manifold $\left(M^8,g\right)$ and a $3$-form $\rho$ such that there exists an oriented orthonormal frame $\left(e_1,\ldots,e_8\right)$ realizing (\ref{eqn:1}) at every point of $M^8$. Here and henceforth we identify $TM^8$ with its dual space ${TM^8}^{\ast}$ using $g$. We will call $\left(e_1,\ldots,e_8\right)$ an \emph{adapted frame} and $\rho$ the \emph{fundamental form} of the $\PSU(3)$-structure.

Consider a $\PSU(3)$-structure $\left(M^8,g,\rho\right)$. We denote by $\left(e_1,\ldots,e_8\right)$ a corresponding adapted frame from now on. The connection forms
\beqn
\omega_{ij}^g:=g\left(\nabla^{g}e_i,e_j\right)
\eeqn
of the Levi-Civita connection $\nabla^g$ define a $1$-form 
\beqn
\Omega^g:=\left(\omega_{ij}^g\right)_{1\leq i,j\leq 8}
\eeqn
with values in the Lie algebra $\so(8)$. We define the \emph{intrinsic torsion} $\Gamma$ of $\left(M^8,g,\rho\right)$ as
\beqn
\Gamma := \pr_\m\left(\Omega^g\right).
\eeqn
Since the Riemannian covariant derivative of the fundamental form $\rho$ is given by
\beqn
\nabla^{g}\rho = \vrho_\ast\left(\Gamma\right)\left(\rho\right),
\eeqn
$\PSU(3)$-structures can be classified by the algebraic type of $\Gamma$ with respect to the decomposition of $\R^8\ox\m$ into irreducible $\PSU(3)$-modules (cf.\ \cite{Fri03}). Applying theorem \ref{thm:1}, we split $\Gamma$ as
\beqn
\Gamma = \Gamma_1+\Gamma_2+\Gamma_3+\Gamma_4+\Gamma_5+\Gamma_6,
\eeqn
to the effect that $64$ classes arise. Via the maps $\Theta_1$ and $\Theta_2$ (see \autoref{sec:2}), we identify the component $\Gamma_1+\Gamma_2+\Gamma_3$ with a $3$-form
\beqn
T^c=T^c_8+T^c_{20}+T^c_{27}\in\Lambda^3_8\op\Lambda^3_{20}\op\Lambda^3_{27}
\eeqn
on $\left(M^8,g,\rho\right)$ and $\Gamma_4+\Gamma_5$ with a $4$-form
\beqn
F^c=F^c_8+F^c_{27}\in\ast\Lambda^4_8\op\ast\Lambda^4_{27}.
\eeqn
We will call $T^c,F^c$ the \emph{characteristic forms} of $\left(M^8,g,\rho\right)$. $\PSU(3)$-structures with $\Gamma=0$ are called \emph{integrable}. We will say that $\left(M^8,g,\rho\right)$ \emph{is of type} $\W_{i_1}\op\ldots\op\W_{i_j}$ if
\beqn
\Gamma\in\W_{i_1}\op\ldots\op\W_{i_j}.
\eeqn
Moreover, $\left(M^8,g,\rho\right)$ \emph{is of strict type} $\W_{i_1}\op\ldots\op\W_{i_j}$ if
the structure is of type $\W_{i_1}\op\ldots\op\W_{i_j}$ and $\Gamma_{i_k}\neq0$ for $k=1,\ldots,j$.
%
%
%
%
\section{Differential equations characterizing the classes}\label{sec:4}\noindent
To begin with, we compute the differential and the co-differential of the fundamental form $\rho$. The differential $d\alpha$ of a $k$-form $\alpha$ on $\left(M^8,g,\rho\right)$ is given by
\beqn
d\alpha=\sum_i e_i\wedge\nabla^{g}_{e_i}\alpha.
\eeqn
Using the formula for the Riemannian covariant derivative of $\rho$, we obtain
\beqn
d\rho=\sum_ie_i\wedge\vrho_\ast\left(\Gamma\left(e_i\right)\right)\left(\rho\right)=\sum_ie_i\wedge\sigma_1\left(\Gamma\left(e_i\right),\rho\right)=:\Pi_1\left(\Gamma\right).
\eeqn
The map
\beqn
\Pi_1:\R^8\ox\m\ra\Lambda^4
\eeqn
is $\PSU(3)$-equivariant. Moreover, we deduce the following:
\begin{lem}
$\Pi_1$ is surjective and the maps $\Pi_1\circ\Theta_1$, $\Pi_1\circ\Theta_2$ satisfy
\bite
\item[i)] $\dsty\Pi_1\circ\Theta_1\left(T\right)=\sigma_+\left(T\right)\quad\all T\in\Lambda^3_8\op\Lambda^3_{27}$,
\item[ii)] $\dsty\Pi_1\circ\Theta_2\left(F\right)=-18\,F\quad\all F\in\ast\Lambda^4_8$,
\item[iii)] $\dsty\Pi_1\circ\Theta_2\left(F\right)=-8\,F\quad\all F\in\ast\Lambda^4_{27}$.
\eite
\end{lem}\noindent
The co-differential $\delta\alpha$ of a $k$-form $\alpha$ on $\left(M^8,g,\rho\right)$ can be computed via
\beqn
\delta\alpha=-\sum_i e_i\hook\nabla^{g}_{e_i}\alpha.
\eeqn
Again, we use the formula for $\nabla^g\rho$:
\beqn
\delta\rho=-\sum_ie_i\hook\vrho_\ast\left(\Gamma\left(e_i\right)\right)\left(\rho\right)=-\sum_ie_i\hook\sigma_1\left(\Gamma\left(e_i\right),\rho\right)=:\Pi_2\left(\Gamma\right).
\eeqn
The map
\beqn
\Pi_2:\R^8\ox\m\ra\Lambda^2
\eeqn
is $\PSU(3)$-equivariant. A direct computation yields here
\begin{lem}
$\Pi_2$ is surjective and the maps $\Pi_2\circ\Theta_1$, $\Pi_2\circ\Theta_2$ satisfy
\bite
\item[i)] $\dsty\Pi_2\circ\Theta_1\left(T\right)=-\sigma_-\left(T\right)\quad\all T\in\Lambda^3_{20}$,
\item[ii)] $\dsty\Pi_2\circ\Theta_2\left(F\right)=-3\,\sigma_+\left(\rho\hook F\right)\quad\all F\in\ast\Lambda^4_8$.
\eite
\end{lem}\noindent
We now summerize the results obtained so far.
\begin{thm}
Let $\left(M^8,g,\rho\right)$ be a $\PSU(3)$-structure with characteristic forms
\begin{align*}
T^c&=T^c_8+T^c_{20}+T^c_{27}, & F^c&=F^c_8+F^c_{27}.
\end{align*}
The differential and the co-differential of $\rho$ are completely determined by $T^c$ and $F^c$:
\begin{align*}
d\rho&=\sigma_+\left(T^c_8+T^c_{27}\right)-18\,F^c_8-8\,F^c_{27}, &
\delta\rho&=-\sigma_-\left(T^c_{20}\right)-3\,\sigma_+\left(\rho\hook F^c_8\right).
\end{align*}
\end{thm}
Based on this theorem, we now deduce differential equations characterizing the algebraic type of $\Gamma$. With the aid of lemma \ref{lem:3}, we compute
\begin{align*}
\delta\rho\hook\rho=\sigma_-\left(\delta\rho\right)&=-\left(\sigma_-\circ\sigma_-\right)\left(T^c_{20}\right)-3\left(\sigma_-\circ\sigma_+\right)\left(\rho\hook F^c_8\right)=-18\left(\rho\hook F^c_8\right),\\
\sigma_+\left(\delta\rho\right)&=-\left(\sigma_+\circ\sigma_-\right)\left(T^c_{20}\right)-3\left(\sigma_+\circ\sigma_+\right)\left(\rho\hook F^c_8\right)=-12\,T^c_{20}
\end{align*}
for the co-differential of the fundamental form. We then examine the differential of the fundamental form in a similar way:
\begin{align*}
\sigma_-\left(d\rho\right)&=\left(\sigma_-\circ\sigma_+\right)\left(T^c_8+T^c_{27}\right)-18\,\sigma_-\left(F^c_8\right)-8\,\sigma_-\left(F^c_{27}\right)\\
&=6\,T^c_8+16\,T^c_{27},\\
\sigma_-\left(\ast d\rho\right)&=\sigma_-\left(\ast\sigma_+\left(T^c_8+T^c_{27}\right)\right)-18\,\sigma_-\left(\ast F^c_8\right)-8\,\sigma_-\left(\ast F^c_{27}\right)\\
&=-18\,\sigma_-\left(\ast F^c_8\right)-8\,\sigma_-\left(\ast F^c_{27}\right).
\end{align*}
Using lemma \ref{lem:4}, we additionally obtain 
\begin{align*}
\rho\wedge\left(\rho\hook d\rho\right)&=\rho\wedge\left(-18\,\rho\hook F^c_8\right)=-180\, F^c_8,\\
\rho\wedge\left(\rho\hook \ast d\rho\right)&=\rho\wedge\left(\rho\hook\ast\sigma_+\left(T^c_8\right)\right)=10\,\ast\sigma_+\left(T^c_8\right).
\end{align*}
All these equations together with lemma \ref{lem:4} prove
\begin{thm}\label{thm:3}
Let $\left(M^8,g,\rho\right)$ be a $\PSU(3)$-structure. The following equivalences hold:
\btab{c|c}
			$\left(M^8,g,\rho\right)$ is of type & $\rho$ satisfies\\
			\hline\hline
			$\W_2\op\W_3\op\W_4\op\W_5\op\W_6$ & $\rho\hook\ast d\rho=0$\\
			\hline
			$\W_1\op\W_3\op\W_4\op\W_5\op\W_6$ & $6\,\delta\rho=\left(\delta\rho\hook\rho\right)\hook\rho$\\
			\hline
			$\W_1\op\W_2\op\W_4\op\W_5\op\W_6$ & $\sigma_-\left(10\,d\rho-\ast\left(\rho\wedge\left(\rho\hook \ast d\rho\right)\right)\right)=0$\\
			\hline
			$\W_1\op\W_2\op\W_3\op\W_5\op\W_6$ & $\delta\rho\hook\rho=0\quad$ or $\quad\rho\hook d\rho=0$\\
			\hline
			$\W_1\op\W_2\op\W_3\op\W_4\op\W_6$ & $\sigma_-\left(10\,\ast d\rho-\ast\left(\rho\wedge\left(\rho\hook d\rho\right)\right)\right)=0$\\
			\hline
			$\W_2\op\W_4\op\W_5\op\W_6$ & $\sigma_-\left(d\rho\right)=0$\\
			\hline
			$\W_1\op\W_2\op\W_3\op\W_6$ & $\sigma_-\left(\ast d\rho\right)=0$\\
			\hline
			$\W_1\op\W_3\op\W_5\op\W_6$ & $\delta\rho=0$\\
			\hline
			$\W_1\op\W_2\op\W_6$ & $10\,\ast d\rho=\rho\wedge\left(\rho\hook \ast d\rho\right)$\\
			\hline
			$\W_2\op\W_4\op\W_6$ & $10\,d\rho=\rho\wedge\left(\rho\hook d\rho\right)$\\
			\hline
			$\W_2\op\W_6$ & $d\rho=0$\\
			\hline
			$\W_6$ & $d\rho=0\quad$ and $\quad\delta\rho=0$
\etab
\end{thm}
Moreover, lemma \ref{lem:5} enables us to express $T^c$ and $F^c$ in terms of $\rho$ as follows:
\begin{prop}
Let
\begin{align*}
T^c&=T^c_8+T^c_{20}+T^c_{27}, & F^c&=F^c_8+F^c_{27}
\end{align*}
be the characteristic forms of a $\PSU(3)$-structure $\left(M^8,g,\rho\right)$. Then
\begin{align*}
T^c_8&=\frac{1}{60}\,\rho\hook\left(\rho\wedge\sigma_-\left(d\rho\right)\right), &
T^c_{20}&=-\frac{1}{12}\,\sigma_+\left(\delta\rho\right), &
T^c_{27}&=\frac{1}{16}\left(\sigma_-\left(d\rho\right)-6\,T^c_8\right),
\end{align*}
\begin{align*}
F^c_8&=-\frac{1}{180}\,\rho\wedge\left(\rho\hook d\rho\right)=-\frac{1}{180}\,\rho\wedge\left(\delta\rho\hook\rho\right), &
F^c_{27}&=-\frac{1}{8}\,\left(d\rho-\sigma_+\left(T^c_8+T^c_{27}\right)+18\,F^c_8\right). 
\end{align*}
\end{prop}
%
%
%
%
\section{The characteristic connection}\label{sec:5}\noindent
Let $\nabla$ be a metric connection on $\left(M^8,g,\rho\right)$. Its torsion $\T$, viewed as a $\left(3,0\right)$-tensor,
\beqn
\T\left(X,Y,Z\right)=g\left(\nabla_XY-\nabla_YX-\left[X,Y\right],Z\right),
\eeqn
is an element of the space
\beqn
\mfr{T}:=\left\{\T\in\Lambda^1\ox\Lambda^1\ox\Lambda^1\setsep\T\left(X,Y,Z\right)+\T\left(Y,X,Z\right)=0\right\}.
\eeqn
$\mfr{T}$ splits into three irreducible $\Orth(8)$-modules (see \cite{Car25}),
\beqn
\mfr{T}=\mfr{T}_8\op\mfr{T}_{56}\op\mfr{T}_{160},
\eeqn
where
\beqn
\mfr{T}_8:=\left\{\T\in\mfr{T}\setsep\ex V\in TM^8:\T\left(X,Y,Z\right)=g\left(g\left(V,X\right)Y-g\left(V,Y\right)X,Z\right)\right\},
\eeqn
$\mfr{T}_{56}:=\Lambda^3$ is the space of $3$-forms on $\left(M^8,g,\rho\right)$ and
\beqn
\mfr{T}_{160}:=\left\{\T\in\mfr{T}\setsep\mfr{S}_{X,Y,Z}\T\left(X,Y,Z\right)=0\ \text{and}\ \textstyle{\sum_i}\T\left(X,e_i,e_i\right)=0\right\}.
\eeqn
Here $\mfr{S}_{X,Y,Z}$ denotes the cyclic sum over $X,Y,Z$. We will say that $\T$ is \emph{totally skew-symmetric} if $\T\in\mfr{T}_{56}$. If $\T\in\mfr{T}_8$, we will call $\T$ \emph{vectorial}. We will furthermore say that $\T$ is \emph{cyclic} if $\T\in\mfr{T}_8\op\mfr{T}_{160}$.

We will now consider the metric connection $\nabla$ on $\left(M^8,g,\rho\right)$ defined by
\beqn
g\left(\nabla_XY,Z\right)=g\left(\nabla^{g}_XY,Z\right)+\frac{1}{2}\, T\left(X,Y,Z\right)+\left(\left(X\hook\rho\right)\hook F\right)\left(Y,Z\right),
\eeqn
where $T\in\Lambda^3_8\op\Lambda^3_{20}\op\Lambda^3_{27}$ and $ F\in\ast\Lambda^4_8\op\ast\Lambda^4_{27}$. Since
\beqn
\mfr{S}_{X,Y,Z}\left(\left(X\hook\rho\right)\hook F\right)\left(Y,Z\right)=0
\eeqn
holds for all $F\in\ast\Lambda^4_8\op\ast\Lambda^4_{27}$, the torsion tensor $\T$ of $\nabla$ is given by
\beqn
\T\left(X,Y,Z\right)=T\left(X,Y,Z\right)-\left(\left(Z\hook\rho\right)\hook F\right)\left(X,Y\right).
\eeqn
We compute
\beqn
\mfr{S}_{X,Y,Z}\T\left(X,Y,Z\right)=\mfr{S}_{X,Y,Z}T\left(X,Y,Z\right)
\eeqn
and
\beqn
\sum_i\T\left(X,e_i,e_i\right)=\sum_i\left(\left(e_i\hook\rho\right)\hook F\right)\left(e_i,X\right)=3\left(\rho\hook F\right)\left(X\right).
\eeqn
Therefore, the projection $\T_{160}$ of $\T$ onto $\mfr{T}_{160}$ is given by
\begin{align*}
\T_{160}\left(X,Y,Z\right)&=\T\left(X,Y,Z\right)-\frac{3}{7}\,g\left(\left(\rho\hook F\right)\left(X\right)Y-\left(\rho\hook F\right)\left(Y\right)X,Z\right)-T\left(X,Y,Z\right)\\
&=-\left(\left(Z\hook\rho\right)\hook F\right)\left(X,Y\right)-\frac{3}{7}\,g\left(\left(\rho\hook F\right)\left(X\right)Y-\left(\rho\hook F\right)\left(Y\right)X,Z\right).
\end{align*}
We then compute
\beqn
\sum_{i,j,k}\T_{160}\left(e_i,e_j,e_k\right)^2=\frac{4}{25}\left\|10\,F-\rho\wedge\left(\rho\hook F\right)\right\|^2+\frac{36}{35}\left\|\rho\hook F\right\|^2.
\eeqn
These equations together with lemma \ref{lem:4} prove
\begin{lem}\label{lem:7}
The torsion tensor $\T$ of $\nabla$ vanishes if and only if $T=0$ and $F=0$. Moreover, $\T\neq0$ is
\bite
\item[i)] totally skew-symmetric if and only if $F=0$.
\item[ii)] cyclic if and only if $T=0$.
\item[iii)] not vectorial.
\item[iv)] an element of $\mfr{T}_{160}$ if and only if $T=0$ and $F\in\ast\Lambda^4_{27}$.
\eite
\end{lem}\noindent
The connection forms
\beqn
\omega_{ij}:=g\left(\nabla e_i,e_j\right)
\eeqn
of the connection $\nabla$ define a $1$-form
\beqn
\Omega:=\left(\omega_{ij}\right)_{1\leq i,j\leq 8}.
\eeqn
Since
\beqn
\Omega\left(X\right)=\Omega^g\left(X\right)+\frac{1}{2}\left(X\hook T\right)+\left(\left(X\hook \rho\right)\hook F\right),
\eeqn
we see that $\Omega$ takes values in the Lie algebra $\so(8)$. We project onto $\m$:
\begin{align*}
\pr_{\m}\left(\Omega\left(X\right)\right)&=\Gamma\left(X\right)+\frac{1}{2}\,\pr_{\m}\left(X\hook T\right)+\pr_{\m}\left(\left(X\hook \rho\right)\hook F\right)\\
&=\Gamma_6\left(X\right)+\frac{1}{2}\,\pr_{\m}\left(X\hook\left(T-T^c\right)\right)
+\pr_{\m}\left(\left(X\hook \rho\right)\hook \left(F-F^c\right)\right).
\end{align*}
Applying lemmata \ref{lem:1} and \ref{lem:2} leads to 
\begin{thm}\label{thm:2}
Let $\left(M^8,g,\rho\right)$ be a $\PSU(3)$-structure with characteristic forms $T^c,F^c$. The metric connection $\nabla$ determined by
\beqn
g\left(\nabla_XY,Z\right)=g\left(\nabla^{g}_XY,Z\right)+\frac{1}{2}\, T\left(X,Y,Z\right)+\left(\left(X\hook\rho\right)\hook F\right)\left(Y,Z\right),
\eeqn
\begin{align*}
T&\in\Lambda^3_8\op\Lambda^3_{20}\op\Lambda^3_{27}, & F&\in\ast\Lambda^4_8\op\ast\Lambda^4_{27}
\end{align*}
preserves the $\PSU(3)$-structure if and only if $\left(M^8,g,\rho\right)$ is of type $\W_1\op\ldots\op\W_5$ and
\begin{align*}
T&=T^c, & F&=F^c.
\end{align*}
\end{thm}
From now on we suppose that $\left(M^8,g,\rho\right)$ is of type
\beqn
\W_1\op\W_2\op\W_3\op\W_4\op\W_5
\eeqn
and define the metric connection $\nabla^c$ on $\left(M^8,g,\rho\right)$ via
\beqn
g\left(\nabla^c_XY,Z\right)=g\left(\nabla^{g}_XY,Z\right)+\frac{1}{2}\, T^c\left(X,Y,Z\right)+\left(\left(X\hook\rho\right)\hook F^c\right)\left(Y,Z\right).
\eeqn
This connection preserves the underlying $\PSU(3)$-structure, i.e.\ $\nabla^c\rho=0$. Moreover, it is completely determined by $\nabla^g$ and the characteristic forms $T^c,F^c$. Therefore, we will call $\nabla^c$ the \emph{characteristic connection} of $\left(M^8,g,\rho\right)$. The respective torsion tensor $\T^c$ will be called \emph{characteristic torsion} of $\left(M^8,g,\rho\right)$.
\begin{cor}\label{cor:2}
Let $\left(M^8,g,\rho\right)$ be a $\PSU(3)$-structure of type $\W_1\op\ldots\op\W_5$. The characteristic torsion $\T^c$ of $\left(M^8,g,\rho\right)$ vanishes if and only if $\left(M^8,g,\rho\right)$ is integrable. Moreover, $\T^c\neq0$ is
\bite
\item[i)] totally skew-symmetric if and only if $\left(M^8,g,\rho\right)$ is of type $\W_1\op\W_2\op\W_3$.
\item[ii)] cyclic if and only if $\left(M^8,g,\rho\right)$ is of type $\W_4\op\W_5$.
\item[iii)] not vectorial.
\item[iv)] an element of $\mfr{T}_{160}$ if and only if $\left(M^8,g,\rho\right)$ is of type $\W_5$.
\eite
\end{cor}\noindent
The tensor field $\A^c$ defined by
\begin{align*}
\A^c\left(X,Y,Z\right)&=g\left(\nabla^c_XY,Z\right)-g\left(\nabla^g_XY,Z\right)\\
&=\frac{1}{2}\, T^c\left(X,Y,Z\right)+\left(\left(X\hook\rho\right)\hook F^c\right)\left(Y,Z\right)
\end{align*}
is an element of the space
\beqn
\mfr{A}:=\left\{\A\in\Lambda^1\ox\Lambda^1\ox\Lambda^1\setsep\A\left(X,Y,Z\right)+\A\left(X,Z,Y\right)=0\right\}.
\eeqn
There exists an $\Orth(8)$-equivariant bijection between $\mfr{A}$ and $\mfr{T}$ (see \cite{Car25}) explicitly given by
\begin{align*}
\mca{T}^c\left(X,Y,Z\right)&=\A^c\left(X,Y,Z\right)-\A^c\left(Y,X,Z\right),\\
2\,\A^c\left(X,Y,Z\right)&=\mca{T}^c\left(X,Y,Z\right)-\mca{T}^c\left(Y,Z,X\right)+\mca{T}^c\left(Z,X,Y\right).
\end{align*}
Since $\nabla^c$ preserves the $\PSU(3)$-structure, applying lemmata \ref{lem:6}, \ref{lem:2} and \ref{lem:7} we have 
\begin{prop}\label{prop:1}
The following are equivalent on $\PSU(3)$-structures $\left(M^8,g,\rho\right)$ of type $\W_1\op\ldots\op\W_5$:
\bite
\item[i)] The characteristic torsion $\T^c$ is $\nabla^c$-parallel.
\item[ii)] The tensor field $\A^c$ is $\nabla^c$-parallel.
\item[iii)] The characteristic forms $T^c,F^c$ are $\nabla^c$-parallel.
\eite
\end{prop}\noindent
If one of these cases holds, we will call $\left(M^8,g,\rho\right)$ a $\PSU(3)$-structure \emph{with parallel characteristic torsion}.

We now intend to compute necessary conditions for the characteristic forms in the presence of parallel characteristic torsion. We investigate the differential $d^c\alpha$ and co-differential $\delta^c\alpha$ of an arbitrary $k$-form $\alpha$ on $\left(M^8,g,\rho\right)$ with respect to $\nabla^c$,
\begin{align*}
d^c\alpha &:= \sum_ie_i\wedge\nabla^c_{e_i}\alpha, &
\delta^c\alpha &:= -\sum_ie_i\hook\nabla^c_{e_i}\alpha.
\end{align*}
We begin with
\begin{align*}
d^c\alpha\left(X_0,\sdots,X_k\right)&=\sum_i\left(e_i\wedge\nabla^c_{e_i}\alpha\right)\left(X_0,\sdots,X_k\right)\\
&=\sum_i\sum_j\left(-1\right)^je_i\left(X_j\right)\cdot\nabla^c_{e_i}\alpha\left(X_0,\sdots,\hat{X_j},\sdots,X_k\right)\\
&=\sum_j\left(-1\right)^j\nabla^c_{X_j}\alpha\left(X_0,\sdots,\hat{X_j},\sdots,X_k\right).
\end{align*}
The formula
\beqn
\nabla_X\alpha\left(X_1,\sdots,X_k\right)=X\left(\alpha\left(X_1,\sdots,X_k\right)\right)-\alpha\left(\nabla_XX_1,X_2,\sdots,X_k\right)-\ldots-\alpha\left(X_1,\sdots,X_{k-1},\nabla_XX_k\right),
\eeqn
valid for any covariant derivative $\nabla$, leads to
\begin{align*}
d^c\alpha\left(X_0,\sdots,X_k\right)&=\sum_j\left(-1\right)^jX_j\left(\alpha\left(X_0,\sdots,\hat{X_j},\sdots,X_k\right)\right)-\sum_{j\neq0}\left(-1\right)^j\alpha\left(\nabla^c_{X_j}X_0,X_1,\sdots,\hat{X_j},\sdots,X_k\right)\\
&\phantom{=\,}-\ldots-\sum_{j\neq k}\left(-1\right)^j\alpha\left(X_0,\sdots,\hat{X_j},\sdots,X_{k-1},\nabla^c_{X_j}X_k\right)\\
&=\sum_j\left(-1\right)^j\nabla^g_{X_j}\alpha\left(X_0,\sdots,\hat{X_j},\sdots,X_k\right)\\
&\phantom{=\,}-\sum_{j\neq0}\sum_l\left(-1\right)^j\A^c\left(X_j,X_0,e_l\right)\cdot\alpha\left(e_l,X_1,\sdots,\hat{X_j},\sdots,X_k\right)\\
&\phantom{=\,}-\ldots-\sum_{j\neq k}\sum_l\left(-1\right)^j\A^c\left(X_j,X_k,e_l\right)\cdot\alpha\left(X_0,\sdots,\hat{X_j},\sdots,X_{k-1},e_l\right).
\end{align*}
Therefore,
\beqn
\left(d^c\alpha-d\alpha\right)\left(X_0,\sdots,X_k\right)=\sum_l\sum_{i<j}\left(-1\right)^{i+j}\T^c\left(X_i,X_j,e_l\right)\cdot\alpha\left(e_l,X_0,\sdots,\hat{X_i},\sdots,\hat{X_j},\sdots, X_k\right).
\eeqn
Finally, we express the right-hand side in terms of the characteristic forms:
\begin{align*}
\left(d^c\alpha-d\alpha\right)\left(X_0,\sdots,X_k\right)&=\sum_l\sum_{i<j}\left(-1\right)^{i+j}\Big(\left(e_l\hook T^c\right)\left(X_i,X_j\right)\\
&\phantom{\sum_l\sum_{i<j}}-\left(\left(e_l\hook\rho\right)\hook F^c\right)\left(X_i,X_j\right)\Big)
\cdot\left(e_l\hook\alpha\right)\left(X_0,\sdots,\hat{X_i},\sdots,\hat{X_j},\sdots, X_k\right)\\
&=-\sum_l\left(e_l\hook T^c\right)\wedge\left(e_l\hook \alpha\right)\left(X_0,\sdots,X_k\right)\\
&\phantom{=\,}+\sum_l\left(\left(e_l\hook\rho\right)\hook F^c\right)\wedge\left(e_l\hook \alpha\right)\left(X_0,\sdots,X_k\right).
\end{align*}
\begin{prop}\label{prop:2}
Let $\left(M^8,g,\rho\right)$ be a $\PSU(3)$-structure of type $\W_1\op\ldots\op\W_5$ with characteristic forms $T^c$, $F^c$, and consider a differential form $\alpha$ on $\left(M^8,g,\rho\right)$. Then
\beqn
d^c\alpha=d\alpha-\sum_i\left(e_i\hook T^c\right)\wedge\left(e_i\hook \alpha\right)+\sum_i\left(\left(e_i\hook\rho\right)\hook F^c\right)\wedge\left(e_i\hook \alpha\right).
\eeqn
\end{prop}\noindent
We continue with
\begin{align*}
\delta^c\alpha\left(X_1,\sdots,X_{k-1}\right)&=-\sum_i\nabla^c_{e_i}\alpha\left(e_i,X_1,\sdots,X_{k-1}\right)\\
&=-\sum_ie_i\left(\alpha\left(e_i,X_1,\sdots,X_{k-1}\right)\right)+\sum_i\alpha\left(\nabla^c_{e_i}e_i,X_1,\sdots,X_{k-1}\right)\\
&\phantom{=\,}+\sum_i\alpha\left(e_i,\nabla^c_{e_i}X_1,\sdots,X_{k-1}\right)+\ldots+\sum_i\alpha\left(e_i,X_1,\sdots,\nabla^c_{e_i}X_{k-1}\right)\\
&=-\sum_i\nabla^g_{e_i}\alpha\left(e_i,X_1,\sdots,X_{k-1}\right)+\sum_{i,j}\A^c\left(e_i,e_i,e_j\right)\cdot\alpha\left(e_j,X_1,\sdots,X_{k-1}\right)\\
&\phantom{=\,}+\sum_{i,j}\A^c\left(e_i,X_1,e_j\right)\cdot\alpha\left(e_i,e_j,X_2,\sdots,X_{k-1}\right)\\
&\phantom{=\,}+\ldots+\sum_{i,j}\A^c\left(e_i,X_{k-1},e_j\right)\cdot\alpha\left(e_i,X_1,\sdots,X_{k-2},e_j\right).
\end{align*}
Since
\beqn
\A^c\left(X,X,Y\right)=\mca{T}^c\left(Y,X,X\right),
\eeqn
we have
\begin{align*}
\left(\delta^c\alpha-\delta\alpha\right)\left(X_1,\sdots,X_{k-1}\right)&=\sum_{i,j}\T^c\left(e_j,e_i,e_i\right)\cdot\alpha\left(e_j,X_1,\sdots,X_{k-1}\right)\\
&\phantom{=\,}+\sum_{i<j}\sum_l\left(-1\right)^l\T^c\left(e_i,e_j,X_l\right)\cdot\alpha\left(e_i,e_j,X_1,\sdots,\hat{X_l},\sdots,X_{k-1}\right).
\end{align*}
Again, we express the difference in terms of $T^c$ and $F^c$:
\begin{align*}
\left(\delta^c\alpha-\delta\alpha\right)\left(X_1,\sdots,X_{k-1}\right)&=\sum_j\sum_i\left(\left(e_i\hook\rho\right)\hook F^c\right)\left(e_i,e_j\right)\cdot\alpha\left(e_j,X_1,\sdots,X_{k-1}\right)\\
&\phantom{=\,}+\sum_{i<j}\sum_l\left(-1\right)^l\Big(\left(e_j\hook e_i\hook T^c\right)\left(X_l\right)+\left(e_j\hook\left(\left(e_i\hook\rho\right)\hook F^c\right)\right)\left(X_l\right)\\
&\phantom{=+}-\left(e_i\hook\left(\left(e_j\hook\rho\right)\hook F^c\right)\right)\left(X_l\right)\Big)\cdot\left(e_j\hook e_i\hook\alpha\right)\left(X_1,\sdots,\hat{X_l},\sdots,X_{k-1}\right)\\
&=3\left(\left(\rho\hook F^c\right)\hook\alpha\right)\left(X_1,\sdots,X_{k-1}\right)\\
&\phantom{=\,}-\frac{1}{2}\sum_{i,j}\left(e_i\hook e_j\hook T^c\right)\wedge\left(e_i\hook e_j\hook\alpha\right)\left(X_1,\sdots,X_{k-1}\right)\\
&\phantom{=\,}-\sum_{i,j}\left(e_i\hook \left(\left(e_j\hook \rho\right)\hook F^c\right)\right)\wedge\left(e_i\hook e_j\hook\alpha\right)\left(X_1,\sdots,X_{k-1}\right).
\end{align*}
\begin{prop}\label{prop:3}
Let $\left(M^8,g,\rho\right)$ be a $\PSU(3)$-structure of type $\W_1\op\ldots\op\W_5$ with characteristic forms $T^c$, $F^c$, and consider a differential form $\alpha$ on $\left(M^8,g,\rho\right)$. Then
\begin{align*}
\delta^c\alpha&=\delta\alpha+3\left(\left(\rho\hook F^c\right)\hook\alpha\right)-\frac{1}{2}\sum_{i,j}\left(e_i\hook e_j\hook T^c\right)\wedge\left(e_i\hook e_j\hook\alpha\right)\\
&\phantom{=\,}-\sum_{i,j}\left(e_i\hook \left(\left(e_j\hook \rho\right)\hook F^c\right)\right)\wedge\left(e_i\hook e_j\hook\alpha\right).
\end{align*}
\end{prop}\noindent
The above considerations enable us to conclude the following necessary conditions for $T^c$ and $F^c$ in presence of parallel characteristic torsion: 
\begin{cor}\label{cor:3}
Let $\left(M^8,g,\rho\right)$ be a $\PSU(3)$-structure of type $\W_1\op\ldots\op\W_5$ with parallel characteristic torsion, i.e.\ $\nabla^cT^c=0$ and $\nabla^cF^c=0$. Then
\begin{align*}
dT^c&=\sum_i\left(e_i\hook T^c\right)\wedge\left(e_i\hook T^c\right)-\sum_i\left(\left(e_i\hook\rho\right)\hook F^c\right)\wedge\left(e_i\hook T^c\right),\\
\delta T^c&=-3\left(\left(\rho\hook F^c\right)\hook T^c\right)+\sum_{i,j}\left(e_i\hook \left(\left(e_j\hook \rho\right)\hook F^c\right)\right)\wedge\left(e_i\hook e_j\hook T^c\right),\\
dF^c&=\sum_i\left(e_i\hook T^c\right)\wedge\left(e_i\hook F^c\right)-\sum_i\left(\left(e_i\hook\rho\right)\hook F^c\right)\wedge\left(e_i\hook F^c\right),\\
\delta F^c&=-3\left(\left(\rho\hook F^c\right)\hook F^c\right)+\frac{1}{2}\sum_{i,j}\left(e_i\hook e_j\hook T^c\right)\wedge\left(e_i\hook e_j\hook F^c\right)\\
&\phantom{=\,}+\sum_{i,j}\left(e_i\hook \left(\left(e_j\hook \rho\right)\hook F^c\right)\right)\wedge\left(e_i\hook e_j\hook F^c\right).
\end{align*}
\end{cor}
The curvature $\mca{R}^c$ of $\nabla^c$, viewed as $(4,0)$-tensor, satisfies
\begin{align*}
\mca{R}^c\left(X,Y,Z,V\right)&=g\left(\nabla^c_X\nabla^c_YZ-\nabla^c_Y\nabla^c_XZ-\nabla^c_{\left[X,Y\right]}Z,V\right)\\
&=g\left(\nabla^g_X\nabla^c_YZ-\nabla^g_Y\nabla^c_XZ-\nabla^g_{\left[X,Y\right]}Z,V\right)\\
&\phantom{=\,}+\A^c\left(X,\nabla^c_YZ,V\right)-\A^c\left(Y,\nabla^c_XZ,V\right)-\A^c\left(\left[X,Y\right],Z,V\right).
\end{align*}
Using
\begin{align*}
g\left(\nabla^g_X\nabla^c_YZ,V\right)&=X\left(g\left(\nabla^c_YZ,V\right)\right)-g\left(\nabla^c_YZ,\nabla^g_XV\right)\\
&=g\left(\nabla^g_X\nabla^g_YZ,V\right)+X\left(\A^c\left(Y,Z,V\right)\right)-\A^c\left(Y,Z,\nabla^g_XV\right)\\
&=g\left(\nabla^g_X\nabla^g_YZ,V\right)+X\left(\A^c\left(Y,Z,V\right)\right)-\A^c\left(Y,Z,\nabla^c_XV\right)\\
&\phantom{=\,}+\sum_i\A^c\left(X,V,e_i\right)\A^c\left(Y,Z,e_i\right)
\end{align*}
and
\begin{align*}
\A^c\left(\left[X,Y\right],Z,V\right)&=\A^c\left(\nabla^c_XY-\nabla^c_YX,Z,V\right)-\sum_i\T^c\left(X,Y,e_i\right)\A^c\left(e_i,Z,V\right),
\end{align*}
we conclude
\begin{prop}\label{prop:5}
Let $\left(M^8,g,\rho\right)$ be a $\PSU(3)$-structure of type $\W_1\op\ldots\op\W_5$. The curvature tensor $\mca{R}^c$ and the Riemannian curvature tensor $\mca{R}^g$ of $\left(M^8,g,\rho\right)$ are related via
\begin{align*}
\mca{R}^c\left(X,Y,Z,V\right)&=\mca{R}^g\left(X,Y,Z,V\right)+\nabla^c_X\A^c\left(Y,Z,V\right)-\nabla^c_Y\A^c\left(X,Z,V\right)\\
&\phantom{=\,}+\sum_i\left(\T^c\left(X,Y,e_i\right)\A^c\left(e_i,Z,V\right)+\A^c\left(X,V,e_i\right)\A^c\left(Y,Z,e_i\right)\right.\\
&\phantom{=+\sum_i\big(\,}\left.-\A^c\left(Y,V,e_i\right)\A^c\left(X,Z,e_i\right)\right).
\end{align*}
\end{prop}\noindent
In the case of parallel characteristic torsion, we compute
\begin{align*}
\mfr{S}_{X,Y,Z}\mca{R}^c\left(X,Y,Z,V\right)&=\sum_i\left(\T^c\left(X,Y,e_i\right)\T^c\left(e_i,Z,V\right)+\T^c\left(Y,Z,e_i\right)\T^c\left(e_i,X,V\right)\right.\\
&\phantom{=\sum_i\big(\,}\left.+\T^c\left(Z,X,e_i\right)\T^c\left(e_i,Y,V\right)\right).
\end{align*}
\begin{cor}\label{cor:1}
Let $\left(M^8,g,\rho\right)$ be a $\PSU(3)$-structure of type $\W_1\op\ldots\op\W_5$ with parallel characteristic torsion, i.e.\ $\nabla^cT^c=0$ and $\nabla^cF^c=0$. Then
\begin{align*}
\mfr{S}_{X,Y,Z}\mca{R}^c\left(X,Y,Z,V\right)&=\sum_i\left(\left(e_i\hook T^c\right)-\left(\left(e_i\hook\rho\right)\hook F^c\right)\right)\wedge\\
&\phantom{=\sum_i(\,}\left(\left(e_i\hook V\hook T^c\right)-\left(e_i\hook\left(\left(V\hook\rho\right)\hook F^c\right)\right)\right)\left(X,Y,Z\right).
\end{align*}
\end{cor}
%
%
%
%
\section{Restricting the characteristic holonomy}\noindent
The holonomy algebra $\hol\left(\nabla^c\right)$ of the characteristic connection (the \emph{characteristic holonomy}) is a Lie subalgebra of $\psu\left(3\right)$ defined up to the adjoint action of $\PSU\left(3\right)$.
\begin{lem}\label{lem:8}
Under the adjoint action of $\PSU\left(3\right)$, any maximal Lie subalgebra of $\psu(3)$ is conjugate to either
\beqn
\R\op\su_c\left(2\right)=\mrm{span}\left(\omega_5,\omega_6,\omega_7,\omega_8\right)
\eeqn
or
\beqn
\so\left(3\right)=\mrm{span}\left(\omega_1,\omega_4,\omega_5\right).
\eeqn
\end{lem}\noindent
The Lie algebra
\beqn
\su_c\left(2\right)=\mrm{span}\left(\omega_5,\omega_6,\omega_7\right)
\eeqn
is the centralizer of
\beqn
\su(2)= \mrm{span}\left(e_{13}+e_{24},e_{14}-e_{23},e_{12}-e_{34}\right)
\eeqn
inside $\g_2\subset\spin\left(7\right)\subset\so\left(8\right)$ (cf.\ \cite{Puh09}), which is isomorphic, but not conjugate, to $\su(2)$. A maximal torus in $\psu\left(3\right)$ is
\beqn
\tor^2=\mrm{span}\left(\omega_7,\omega_8\right).
\eeqn
The $1$-dimensional Lie algebra generated by $\omega_7$ will be denoted by $\tor^1$.

Up to a factor, $\rho$ is the only $\psu\left(3\right)$-invariant differential form on $\left(M^8,g,\rho\right)$. Applying proposition \ref{prop:1} this proves
\begin{prop}\label{prop:6}
Let $\left(M^8,g,\rho\right)$ be a $\PSU\left(3\right)$-structure of type $\W_1\op\ldots\op\W_5$ with parallel characteristic torsion and $\hol\left(\nabla^c\right)=\psu\left(3\right)$. Then $\left(M^8,g,\rho\right)$ is integrable, i.e.\ $\Gamma=0$.
\end{prop}

The space $\mca{K}\left(\h\right)$ of algebraic curvature tensors with values in $\h$,
\beqn
\mca{K}\left(\h\right):=\left\{\mca{R}\in\Lambda^2\ox\h\,\,:\,\,
\mfr{S}_{X,Y,Z}\mca{R}\left(X,Y,Z,V\right)=0\right\},
\eeqn
is trivial for $\h=\R\op\su_c\left(2\right),\so\left(3\right)$. Lemma 5.6 in \cite{CS04} therefore leads to
\begin{prop}\label{prop:4}
Let $\left(M^8,g,\rho\right)$ be a non-integrable $\PSU\left(3\right)$-structure of type
\beqn
\W_1\op\W_2\op\W_3\op\W_4\op\W_5
\eeqn
with parallel characteristic torsion $\T^c$. Then its characteristic connection $\nabla^c$ is an Ambrose-Singer connection, i.e.\
\beqn
\nabla^c\T^c=0,\quad\nabla^c\mca{R}^c=0.
\eeqn
\end{prop}\noindent
With \cites{AS58,TV83} we immediately deduce
\begin{cor}
Let $\left(M^8,g,\rho\right)$ be a simply-connected and complete $\PSU\left(3\right)$-manifold of type $\W_1\op\ldots\op\W_5$ with non-zero parallel characteristic torsion. Then $\left(M^8,g\right)$ is homogeneous.
\end{cor}
By lemma \ref{lem:8} $\hol\left(\nabla^c\right)\subset\psu\left(3\right)$ is one of
\beqn
\R\op\su_c\left(2\right),\quad\su_c\left(2\right),\quad\tor^2,\quad\so(3),\quad\tor^1,\quad\zero,
\eeqn
where $\zero$ denotes the zero algebra. We discuss the first four cases in detail.
%
%
\subsection{The cases \texorpdfstring{$\hol\left(\nabla^c\right)=\R\op\su_c\left(2\right)$, $\su_c\left(2\right)$, $\tor^2$}{}}
Here, the $4$-form 
\beqn
\Phi:=\phi+\ast\phi,\quad\phi:=\left(e_{246}-e_{235}-e_{145}-e_{136}+e_{127}+e_{347}+e_{567}\right)\wedge e_8,
\eeqn
is globally well defined and $\nabla^c$-parallel. Consequently, $\left(M^8,g\right)$ admits a $\Spin(7)$-structure with fundamental form $\Phi$. We recall certain facts on these structures. Any $\Spin\left(7\right)$-manifold $\left(M^8,g,\Phi\right)$ admits (see \cite{Iva04}) a unique metric connection $\tilde{\nabla}^c$ with totally skew-symmetric torsion $\tilde{T}^c$ preserving the structure, i.e.\
\beqn
\tilde{\nabla}^c\Phi=0.
\eeqn
In Cabrera's description \cite{Cab95} of the Fern\'andez classification \cite{Fer86} a $\Spin\left(7\right)$-structure is \emph{balanced} if and only if the Lee form
\beqn
\theta:=\frac{1}{7}\,\ast\left(\delta\Phi\wedge\Phi\right)
\eeqn
vanishes. Those for which $d\Phi=\theta\wedge\Phi$ holds are \emph{locally conformal parallel}.
\begin{thm}\label{thm:4}
Let $\left(M^8,g,\rho\right)$ be a $\PSU(3)$-manifold of type $\W_1\op\ldots\op\W_5$ with $\hol\left(\nabla^c\right)$ one of
\beqn
\R\op\su_c\left(2\right),\quad\su_c\left(2\right),\quad\tor^2.
\eeqn
Then $\left(M^8,g,\rho\right)$ admits a $\Spin(7)$-structure $\left(M^8,g,\Phi,\tilde{\nabla}^c,\tilde{T}^c\right)$ preserved by $\nabla^c$. The connection $\tilde{\nabla}^c$ coincides with $\nabla^c$ if $\left(M^8,g,\rho\right)$ is of type $\W_1\op\W_2\op\W_3$. Conversely, any $\Spin\left(7\right)$-manifold $\left(M^8,g,\Phi,\tilde{\nabla}^c,\tilde{T}^c\right)$ with $\hol\left(\tilde{\nabla}^c\right)=\R\op\su_c\left(2\right)$, $\su_c\left(2\right)$ or $\tor^2$ admits a $\PSU(3)$-structure of type $\W_1\op\W_2\op\W_3$ with characteristic torsion
\beqn
\pr_{\Lambda^3_8\op\Lambda^3_{20}\op\Lambda^3_{27}}\left(\tilde{T}^c\right).
\eeqn
\end{thm}
\begin{proof}
There remains to show the second part of the theorem. As $\hol\left(\tilde{\nabla}^c\right)$ is one of
\beqn
\R\op\su_c\left(2\right),\quad \su_c\left(2\right),\quad \tor^2,
\eeqn
the $3$-form $\rho$ is globally well defined and $\tilde{\nabla}^c$-parallel. The latter yields
\beqn
\sigma_1\left(\Gamma\left(X\right),\rho\right)=\nabla^g_X\rho=\sigma_1\left(-\frac{1}{2}\left(X\hook \tilde{T}^c\right),\rho\right)
\eeqn
for the intrinsic torsion $\Gamma$ of the corresponding $\PSU\left(3\right)$-structure. Lemma \ref{lem:3} completes the proof.
\end{proof}
We now discuss the case of parallel characteristic torsion, i.e.\ $\nabla^c\T^c=0$. Propositions \ref{prop:1} and \ref{prop:4} show that
\beqn
T^c\in\Lambda^3_8\op\Lambda^3_{20}\op\Lambda^3_{27},\quad F^c\in\ast\Lambda^4_8\op\ast\Lambda^4_{27},\quad \mca{R}^c:\Lambda^2\ra\hol\left(\nabla^c\right)
\eeqn
are $\hol\left(\nabla^c\right)$-invariant. Moreover, by corollary \ref{cor:1} we have
\begin{align*}
\mfr{S}_{X,Y,Z}\mca{R}^c\left(X,Y,Z,V\right)&=\sum_i\left(\left(e_i\hook T^c\right)-\left(\left(e_i\hook\rho\right)\hook F^c\right)\right)\wedge\\
&\phantom{=\sum_i(\,}\left(\left(e_i\hook V\hook T^c\right)-\left(e_i\hook\left(\left(V\hook\rho\right)\hook F^c\right)\right)\right)\left(X,Y,Z\right).
\end{align*}
We denote by $\left(\bullet\right)$ the system of equations described above. 
\begin{thm}\label{thm:5}
Let $\left(M^8,g,\rho\right)$ be a $\PSU(3)$-structure of type $\W_1\op\ldots\op\W_5$ with parallel characteristic torsion and $\hol\left(\nabla^c\right)$ one of
\beqn
\R\op\su_c\left(2\right),\quad\su_c\left(2\right),\quad\tor^2.
\eeqn
Then $\left(M^8,g,\rho\right)$ is of type $\W_1\op\W_2\op\W_3$.
\end{thm}
\begin{proof}
Assume $F^c\neq0$. Analyzing $\left(\bullet\right)$ for the cases $\hol\left(\nabla^c\right)=\R\op\su_c\left(2\right),\su_c\left(2\right),\tor^2$ there exist only one solution,
\begin{align*}
T^c&=0,\\
F^c&=b\left(e_{246}-e_{235}-e_{145}-e_{136}+e_{127}+e_{347}+e_{567}\right)\wedge e_8,\\
\mca{R}^c&=-6\,b^2\left(e_{67}\ox\omega_5-e_{57}\ox\omega_6+e_{56}\ox\omega_7\right),
\end{align*}
where $b\in\R\backslash\left\{0\right\}$. Now, proposition \ref{prop:5} and the Ambrose-Singer holonomy theorem enable us to determine the Riemannian holonomy algebra. Up to the adjoint action of $\SO\left(8\right)$, we compute $\hol\left(\nabla^g\right)=\psu(3)$, hence
\beqn
0=\nabla^g_X\rho=\vrho_\ast\left(\Gamma\left(X\right)\right)\left(\rho\right)=-\sigma_1\left(\left(X\hook\rho\right)\hook F^c,\rho\right),
\eeqn
and lemmata \ref{lem:3} and \ref{lem:6} yield $F^c=0$, a contradiction.
\end{proof}\noindent
We therefore conclude
\beqn
\tilde{\nabla}^c=\nabla^c,\quad \tilde{T}^c=T^c,\quad F^c=0.
\eeqn

Assume $\hol\left(\nabla^c\right)=\R\op\su_c\left(2\right)$. In this case, the following forms are globally well defined and $\nabla^c$-parallel
\beqn
e_8,\quad \omega_8,\quad \vphi_1:=e_{246}-e_{235}-e_{145}-e_{136}+e_{127}+e_{347},\quad \vphi_2:=e_{567}.
\eeqn
Solving $\left(\bullet\right)$ yields
\begin{align*}
T^c&=a_1\left(\vphi_1+3\,\vphi_2\right)+a_2\left(\omega_8\wedge e_8+3\,\vphi_2\right),\\
F^c&=0,\\
\mca{R}^c&=-\frac{1}{2}\left(5\,a_1^2+3\,a_1\,a_2\right)(\omega_5\ox\omega_5+\omega_6\ox\omega_6+\omega_7\ox\omega_7)\\
&\phantom{=\,}-\frac{1}{2}\left(7\,a_1^2+3\,a_1\,a_2-2\,a_2^2\right)\omega_8\ox\omega_8,
\end{align*}
where $a_1,a_2\in\R$ satisfy
\beqn
5\,a_1^2+3\,a_1\,a_2\neq0,\quad 7\,a_1^2+3\,a_1\,a_2-2\,a_2^2\neq0
\eeqn
to ensure $\hol\left(\nabla^c\right)=\R\op\su_c\left(2\right)$.
\begin{prop}
Let $\left(M^8,g,\rho\right)$ be a $\PSU\left(3\right)$-structure of type $\W_1\op\ldots\op\W_5$ with parallel characteristic torsion and $\hol\left(\nabla^c\right)=\R\op\su_c\left(2\right)$. Then $\left(M^8,g,\rho\right)$ is of type $\W_1\op\W_3$ and the corresponding $\Spin(7)$-structure is not locally conformal parallel.
\end{prop}\noindent
The fundamental form $\rho$ can be expressed as
\beqn
\rho=\vphi_1-2\,\vphi_2+\omega_8\wedge e_8.
\eeqn
Using proposition \ref{prop:2}, we compute the following differentials
\begin{align*}
de_8&=a_2\,\omega_8,\\
d\omega_8&=0,\\
d\vphi_1&=\left(7\,a_1+3\,a_2\right)\left(e_8\hook\ast\vphi_1\right)+6\,a_1\left(e_8\hook\ast\vphi_2\right), \\
d\vphi_2&=a_1\left(e_8\hook\ast\vphi_1\right),\\
d\left(e_8\hook\ast\vphi_1\right)&=0, \\
d\left(e_8\hook\ast\vphi_2\right)&=0.
\end{align*}
Consequently, several Lie derivatives along $e_8$ vanish,
\beqn
\Lie_{e_8}\omega_8=0,\quad \Lie_{e_8}\vphi_1=0,\quad \Lie_{e_8}\vphi_2=0,\quad
\Lie_{e_8}\left(e_8\hook\ast\vphi_1\right)=0,\quad \Lie_{e_8}\left(e_8\hook\ast\vphi_2\right)=0.
\eeqn
A computation of the Riemannian Ricci tensor yields the following result:
\begin{align*}
\Ric^g\left(e_i\right)&=\frac{1}{2}\left(39\,a_1^2+18\,a_1\,a_2-3\,a_2^2\right)e_i & \text{for}\quad i&=1,\ldots,4,\\
\Ric^g\left(e_i\right)&=\frac{1}{2}\left(51\,a_1^2+42\,a_1\,a_2+9\,a_2^2\right)e_i& \text{for}\quad i&=5,\ldots,7,\\
\Ric^g\left(e_8\right)&=3\,a_2^2\,e_8.
\end{align*}
We now restrict to the regular case, i.e.\ we assume that $e_8$ induces a free action of the group $\mrm{S}^1$. The orbit space $\pi:M^8\ra\overline{N}$ is a Riemannian $7$-manifold $\overline{N}$. There exist well-defined differential forms $\overline{\omega_8}$, $\overline{\vphi_1}$ and $\overline{\vphi_2}$ on $\overline{N}$ such that
\beqn
\overline{T^c}=a_1\left(\overline{\vphi_1}+\overline{\vphi_2}\right)+\left(2\,a_1+3\,a_2\right)\overline{\vphi_2},\quad \overline{\ast}\,\overline{\vphi_1}=\overline{e_8\hook\ast\vphi_1},\quad \overline{\ast}\,\overline{\vphi_2}=\overline{e_8\hook\ast\vphi_2}.
\eeqn
Here $\overline{\ast}$ denotes the Hodge operator of $\overline{N}$. The $3$-form
$\overline{\vphi}:=\overline{\vphi_1}+\overline{\vphi_2}$ satisfies
\beqn
d\overline{\ast}\,\overline{\vphi}=0.
\eeqn
Consequently, $\overline{N}$ is a cocalibrated $\G_2$-manifold with fundamental form $\overline{\vphi}$. Any manifold of this type admits a unique metric connection with totally skew-symmetric torsion (see \cite{FI02}). Since
\beqn
\overline{T^c}=\frac{1}{6}\lan d\overline{\vphi},\overline{\ast}\,\overline{\vphi}\ran\overline{\vphi}-\overline{\ast}d\overline{\vphi},
\eeqn
this connection coincides with $\overline{\nabla^c}$. The holomy algebra of $\overline{\nabla^c}$ is $\su_c\left(2\right)\subset\g_2$ if
\beqn
7\,a_1+3\,a_2=0,
\eeqn
and $\R\op\su_c\left(2\right)\subset\g_2$ otherwise. Conversely, if $(\overline{N},\overline{g},\overline{\vphi},\overline{\nabla^c},\overline{T^c})$ is a cocalibrated $\G_2$-manifold of this type, then the differential forms $\overline{\vphi_1}$, $\overline{\vphi_2}$ and $\overline{\omega_8}$ exist and $d\overline{\omega_8}=0$ holds. Suppose the equation
\beqn
de_8=a_2\,\overline{\omega_8}
\eeqn
defines a principal $\mrm{S}^1$-bundle $\pi:M^8\ra\overline{N}$. Then the manifold $M^8$ admits a $\PSU(3)$-structure
\beqn
\rho=\pi^{\ast}\left(\overline{\vphi_1}\right)-2\,\pi^{\ast}\left(\overline{\vphi_2}\right)+\pi^{\ast}\left(\overline{\omega_8}\right)\wedge e_8
\eeqn
with parallel characteristic torsion
\beqn
T^c=a_1\left(\pi^{\ast}\left(\overline{\vphi_1}\right)+3\,\pi^{\ast}\left(\overline{\vphi_2}\right)\right)+a_2\left(\pi^{\ast}\left(\overline{\omega_8}\right)\wedge e_8+3\,\pi^{\ast}\left(\overline{\vphi_2}\right)\right)
\eeqn
and $\hol\left(\nabla^c\right)\subseteq\R\op\su_c\left(2\right)$.
\begin{thm}\label{thm:6}
Let $\left(M^8,g,\rho\right)$ be a regular $\PSU\left(3\right)$-manifold of type $\W_1\op\ldots\op\W_5$ with parallel characteristic torsion and
\beqn
\hol\left(\nabla^c\right)=\R\op\su_c\left(2\right).
\eeqn
Then $M^8$ is a principal $\mrm{S}^1$-bundle and a Riemannian submersion over a cocalibrated $\G_2$-manifold $(\overline{N},\overline{g},\overline{\vphi},\overline{\nabla^c},\overline{T^c})$ with $\overline{\nabla^c}\,\overline{T^c}=0$ and $\hol\left(\overline{\nabla^c}\right)=\R\op\su_c\left(2\right)$ or $\su_c\left(2\right)$. The Chern class of the fibration $\pi:M^8\ra\overline{N}$ is proportional to the form $\overline{\omega_8}$. Conversely, any of these fibrations admits a $\PSU\left(3\right)$-structure of type $\W_1\op\W_2\op\W_3$ with parallel characteristic torsion and characteristic holonomy contained in $\R\op\su_c\left(2\right)$.
\end{thm}
\begin{exa}\label{exa:1}
There exists a unique simply-connected, complete, cocalibrated $\G_2$-manifold $(\overline{N},\overline{g},\overline{\vphi},\overline{\nabla^c},\overline{T^c})$ with $\overline{\nabla^c}\,\overline{T^c}=0$ and $\hol\left(\overline{\nabla^c}\right)=\su_c\left(2\right)$ (see \cite{Fri07}). This manifold is a naturally reductive homogeneous space. Moreover, $\overline{N}=N(1,1)$ is a nearly parallel $\G_2$-manifold with $\hol\left(\overline{\nabla^c}\right)=\R\op\su_c\left(2\right)$, it appears in the classification \cite{FKMS97}.
\end{exa}
We now consider the case $\hol\left(\nabla^c\right)=\su_c\left(2\right)$. The differential forms 
\beqn
e_8,\quad \omega_8,\quad \vphi_1,\quad \vphi_2
\eeqn
are globally well defined and $\nabla^c$-parallel. Analyzing $\left(\bullet\right)$ yields
\begin{align*}
T^c&=a_1\left(\vphi_1+3\,\vphi_2\right)+a_2\left(\omega_8\wedge e_8+3\,\vphi_2\right)+a_3\left(e_{148}-e_{238}\right)+a_4\left(e_{138}+e_{248}\right),\\
F^c&=0,\\
\mca{R}^c&=-\frac{1}{2}\left(5\,a_1^2+3\,a_1\,a_2\right)(\omega_5\ox\omega_5+\omega_6\ox\omega_6+\omega_7\ox\omega_7),
\end{align*}
where $a_1,a_2,a_3,a_4\in\R$ satisfy
\beqn
5\,a_1^2+3\,a_1\,a_2\neq0,\quad 7\,a_1^2+3\,a_1\,a_2-2\,a_2^2=\frac{2}{3}\left(a_3^2+a_4^2\right).
\eeqn
Applying theorems \ref{thm:4} and \ref{thm:5}, two torsion forms $T^c_1,T^c_2\neq0$ of this family define equivalent geometric structures if they are equivalent under the action of the normalizer of $\Hol\left(\nabla^c\right)$ inside $\Spin(7)$. Without loss of generality, we can therefore assume $a_3=a_4=0$. 
\begin{prop}
Let $\left(M^8,g,\rho\right)$ be a $\PSU\left(3\right)$-structure of type $\W_1\op\ldots\op\W_5$ with parallel characteristic torsion and $\hol\left(\nabla^c\right)=\su_c\left(2\right)$. Then $\left(M^8,g,\rho\right)$ is of strict type $\W_1\op\W_3$, $\Ric^g$ is positive definite and the corresponding $\Spin(7)$-structure is neither balanced nor locally conformal parallel.
\end{prop}\noindent
Following the same arguments that lead to theorem \ref{thm:6}, we obtain
\begin{thm}\label{thm:7}
Let $\left(M^8,g,\rho\right)$ be a regular $\PSU\left(3\right)$-manifold of type $\W_1\op\ldots\op\W_5$ with parallel characteristic torsion and $\hol\left(\nabla^c\right)=\su_c\left(2\right)$. Then $M^8$ is a principal $\mrm{S}^1$-bundle and a Riemannian submersion over a cocalibrated $\G_2$-manifold $(\overline{N},\overline{g},\overline{\vphi},\overline{\nabla^c},\overline{T^c})$ with $\overline{\nabla^c}\overline{T^c}=0$ and $\hol\left(\overline{\nabla^c}\right)=\R\op\su_c\left(2\right)$. The Chern class of the fibration $\pi:M^8\ra\overline{N}$ is proportional to the form $\overline{\omega_8}$.
\end{thm}
We now discuss the case of $\hol\left(\nabla^c\right)=\tor^2$. The following forms are globally well defined and $\nabla^c$-parallel
\beqn
e_7,\quad e_8,\quad \omega_7,\quad\omega_8,\quad \Sigma:=e_{246}-e_{235}-e_{145}-e_{136},\quad \Omega_1:=e_{12}+e_{34},\quad \Omega_2:=e_{56}.
\eeqn
Using these, the fundamental form reads
\beqn
\rho=\Sigma+\omega_7\wedge e_7+\omega_8\wedge e_8.
\eeqn
There exist two $5$-parameter families of triples $\left(T^c,F^c,\mca{R}^c\right)$ satisfying system $\left(\bullet\right)$. Quotienting out the action of the normalizer of $\Hol\left(\nabla^c\right)$ inside $\Spin(7)$, we obtain the two $4$-parameter families
\begin{align*}
T^c_I&=a_1\,\Sigma+\left(\left(a_1+a_2\right)\omega_7+a_3\,\omega_8\right)\wedge e_7+\left(a_4\,\omega_7-\left(\tfrac 5 3\,a_1+a_2\right)\omega_8\right)\wedge e_8,\\
F^c_I&=0,\\
\mca{R}^c_I&=\left(\left(a_1+a_2\right)^2-a_1^2+a_4^2\right)\omega_7\ox\omega_7+\left(\left(a_1+a_2\right)a_3-\left(\tfrac 5 3\,a_1+a_2\right)a_4\right)\left(\omega_7\ox\omega_8+\omega_8\ox\omega_7\right)\\
&\phantom{=\,}+\left(\left(\tfrac 5 3\,a_1+a_2\right)^2-a_1^2+a_3^2\right)\omega_8\ox\omega_8
\end{align*}
and
\begin{align*}
T^c_{I\!I}&= \left(b_1\,\Omega_1+(b_1-3\,b_2)\Omega_2+b_3\,\omega_8\right)\wedge e_7
+(b_4\,\omega_7-b_2\,\omega_8)\wedge e_8,\\
F^c_{I\!I}&=0,\\
\mca{R}^c_{I\!I}&=\left(-\tfrac{1}{2}\,b_1\left(b_1-3\,b_2\right)+b_4^2\right)\omega_7\ox\omega_7+\left(-\tfrac{1}{2}\,b_3\left(b_1-3\,b_2\right)-b_2\,b_4\right)\left(\omega_7\ox\omega_8+\omega_8\ox\omega_7\right)\\
&\phantom{=\,}+\left(-\tfrac{1}{2}\,b_1\left(b_1-b_2\right)+b_2^2+b_3^2\right)\omega_8\ox\omega_8,
\end{align*}
where $a_1,a_2,a_3,a_4\in\R$ and $b_1,b_2,b_3,b_4\in\R$ satisfy
\begin{align*}
\left(\left(a_1+a_2\right)^2-a_1^2+a_4^2\right)\left(\left(\tfrac 5 3\,a_1+a_2\right)^2-a_1^2+a_3^2\right)&\neq \left(\left(a_1+a_2\right)a_3-\left(\tfrac 5 3\,a_1+a_2\right)a_4\right)^2,\\
\left(-\tfrac{1}{2}\,b_1\left(b_1-3\,b_2\right)+b_4^2\right)\left(-\tfrac{1}{2}\,b_1\left(b_1-b_2\right)+b_2^2+b_3^2\right)&\neq \left(-\tfrac{1}{2}\,b_3\left(b_1-3\,b_2\right)-b_2\,b_4\right)^2.
\end{align*}
To distinguish between these two cases, we will say that $\T^c$ \emph{is of type $I$ or of type $I\!I$} respectively.
\begin{prop}
Let $\left(M^8,g,\rho\right)$ be a $\PSU\left(3\right)$-structure of type $\W_1\op\ldots\op\W_5$ with parallel characteristic torsion and $\hol\left(\nabla^c\right)=\tor^2$. Then the corresponding $\Spin(7)$-structure $\left(M^8,g,\Phi\right)$ is not locally conformal parallel, it is balanced if and only if $\T^c$ is of both types.
\end{prop}\noindent
Assume $\T^c$ is of type $I$. With the aid of proposition \ref{prop:2}, we compute
\begin{align*}
de_7&=\left(a_1+a_2\right)\omega_7+a_3\,\omega_8, & de_8&=a_4\,\omega_7-\left(\tfrac 5 3\,a_1+a_2\right)\omega_8,\\
d\omega_7&=0,& d\omega_8&=0,\\
d\Sigma&=4\,a_1\left(e_8\hook\ast\left(\Omega\wedge e_7\right)\right),& d\Omega&=3\,a_1\left(e_8\hook e_7\hook\ast\Sigma\right),
\end{align*}
where $\Omega:=\Omega_1+\Omega_2$. Therefore,
\beqn
\Lie_{e_8}e_7=0,\quad \Lie_{e_8}\omega_7=0,\quad \Lie_{e_8}\omega_8=0,\quad\Lie_{e_8}\Sigma=0,\quad \Lie_{e_8}\Omega=0
\eeqn
and
\beqn
\Lie_{e_8}\left(e_8\hook\ast\left(\Omega\wedge e_7\right)\right)=0,\quad \Lie_{e_8}\left(e_8\hook e_7\hook\ast\Sigma\right)=0.
\eeqn
For a regular structure, $\pi:M^8\ra \overline{N}$ is a principal $\mrm{S}^1$-bundle over a $7$-dimensional Riemannian manifold $\overline{N}$. The latter admits differential forms $\overline{\Sigma}$, $\overline{\Omega}$, $\overline{e_7}$, $\overline{\omega_7}$ and $\overline{\omega_8}$ such that
\begin{align*}
\overline{T^c}&=a_1\,\overline{\Sigma}+\left(\left(a_1+a_2\right)\overline{\omega_7}+a_3\,\overline{\omega_8}\right)\wedge \overline{e_7},\\
\overline{\ast}\,\overline{\Sigma}&=\overline{\left(e_8\hook e_7\hook\ast\Sigma\right)}\wedge\overline{e_7},\\
\overline{\ast}\left(\overline{\Omega}\wedge\overline{e_7}\right)&=\overline{\left(e_8\hook\ast\left(\Omega\wedge e_7\right)\right)}.
\end{align*}
Introducing the form
\beqn
\overline{\vphi}:=\overline{\Sigma}+\overline{\Omega}\wedge\overline{e_7},
\eeqn
we obtain
\begin{align*}
d\overline{\ast}\,\overline{\vphi}&=0,\\
\overline{T^c}&=\frac{1}{6}\lan d\overline{\vphi},\overline{\ast}\,\overline{\vphi}\ran\overline{\vphi}-\overline{\ast}d\overline{\vphi}.
\end{align*}
Again, we deduce that $(\overline{N},\overline{g},\overline{\vphi},\overline{\nabla^c},\overline{T^c})$ is a cocalibrated $\G_2$-manifold with $\overline{\nabla^c}\,\overline{T^c}=0$. The holonomy algebra of $\overline{\nabla^c}$ is contained in $\tor^2\subset\g_2$. $\G_2$-structures of this type admit two $\overline{\nabla^c}$-parallel spinor fields $\overline{\Psi_+}$, $\overline{\Psi_-}$ such that
\beqn
\overline{\Sigma}\cdot\overline{\Psi_{+}}=4\,\overline{\Psi_{+}},\quad \overline{\Sigma}\cdot\overline{\Psi_{-}}=-4\,\overline{\Psi_{-}}.
\eeqn
The '$\cdot$' denotes the Clifford product.
\begin{thm}\label{thm:8}
Let $\left(M^8,g,\rho\right)$ be a regular $\PSU\left(3\right)$-manifold of type $\W_1\op\ldots\op\W_5$ with parallel characteristic torsion $\T^c$ and $\hol\left(\nabla^c\right)=\tor^2$. Moreover, suppose that $\T^c$ is of type $I$. Then $M^8$ is a principal $\mrm{S}^1$-bundle and a Riemannian submersion over a cocalibrated $\G_2$-manifold $(\overline{N},\overline{g},\overline{\vphi},\overline{\nabla^c},\overline{T^c})$ that satisfies
\beqn
\overline{\nabla^c}\,\overline{T^c}=0,\quad \hol\left(\overline{\nabla^c}\right)\subseteq\tor^2\subset\g_2,\quad \overline{T^c}\cdot\overline{\Psi_\pm}=\pm\lambda\overline{\Psi_\pm},\quad \lambda\in\R.
\eeqn
The Chern class of the fibration $\pi:M^8\ra\overline{N}$ is a linear combination of $\overline{\omega_7}$ and $\overline{\omega_8}$.
\end{thm}\noindent
The case of torsion type $I\!I$ is similar. Following the same arguments as in the discussion of type $I$, we obtain
\begin{thm}\label{thm:11}
Let $\left(M^8,g,\rho\right)$ be a regular $\PSU\left(3\right)$-manifold of type $\W_1\op\ldots\op\W_5$ with parallel characteristic torsion $\T^c$ and $\hol\left(\nabla^c\right)=\tor^2$. Moreover, suppose that $\T^c$ is of type $I\!I$. Then $M^8$ is a principal $\mrm{S}^1$-bundle and a Riemannian submersion over a cocalibrated $\G_2$-manifold $(\overline{N},\overline{g},\overline{\vphi},\overline{\nabla^c},\overline{T^c})$ that satisfies
\beqn
\overline{\nabla^c}\,\overline{T^c}=0,\quad \hol\left(\overline{\nabla^c}\right)\subseteq\tor^2\subset\g_2,\quad \overline{T^c}\cdot\overline{\Psi_\pm}=\lambda\overline{\Psi_\pm},\quad \lambda\in\R.
\eeqn
The Chern class of the fibration $\pi:M^8\ra\overline{N}$ is a linear combination of $\overline{\omega_7}$ and $\overline{\omega_8}$.
\end{thm}
%
%
%
%
%
\subsection{The case of \texorpdfstring{$\hol\left(\nabla^c\right)=\so\left(3\right)$}{}}
Any $\so(3)$-invariant $3$-form in $\Lambda^3_8\op\Lambda^3_{20}\op\Lambda^3_{27}$ is a multiple of
\beqn
\rho+16\,e_{145}\in\Lambda^3_{27}.
\eeqn
Up to a factor, the only $\so(3)$-invariant $4$-form in $\ast\Lambda^4_8\op\ast\Lambda^4_{27}$ is
\beqn
\ast\sigma_+\left(\rho+16\,e_{145}\right)\in\ast\Lambda^4_{27}.
\eeqn
Assume $\nabla^c\T^c=0$. This implies
\beqn
T^c=a\left(\rho+16\,e_{145}\right),\quad F^c=b\,\ast\sigma_+\left(\rho+16\,e_{145}\right)
\eeqn
for some $a,b\in\R$ and the curvature operator $\mca{R}^c:\Lambda^2\ra\so(3)$ is $\so(3)$-invariant (see proposition \ref{prop:4}). Moreover, by corollary \ref{cor:1} the triple $(T^c,F^c,\mca{R}^c)$ satisfies
\begin{align*}
\mfr{S}_{X,Y,Z}\mca{R}^c\left(X,Y,Z,V\right)&=\sum_i\left(\left(e_i\hook T^c\right)-\left(\left(e_i\hook\rho\right)\hook F^c\right)\right)\wedge\\
&\phantom{=\sum_i(\,}\left(\left(e_i\hook V\hook T^c\right)-\left(e_i\hook\left(\left(V\hook\rho\right)\hook F^c\right)\right)\right)\left(X,Y,Z\right).
\end{align*}
There exist only two solutions, either $a\neq0$, $b=0$ and
\beqn
\mca{R}^c=-16\,a^2\left(\omega_1\ox\omega_1+\omega_4\ox\omega_4+\omega_5\ox\omega_5\right)
\eeqn
or $a=0$, $b\neq0$ and
\begin{align*}
\mca{R}^c&=-3072\,b^2\Big(\left(e_{36}-e_{27}-\sqrt{3}\,e_{28}\right)\ox\omega_1           +\left(e_{26}+e_{37}-\sqrt{3}\,e_{38}\right)\ox\omega_4\\
&\phantom{=\,-3072\,b^2\big(}+\left(e_{23}+2\,e_{67}\right)\ox\omega_5\Big)
\end{align*}
holds necessarily. The respective Riemannian Ricci tensors are
\begin{align*}
\Ric^g&=3\,a^2\,\diag\left(49,33,33,49,49,33,33,33\right),\\
\Ric^g&=-3072\,b^2\,\diag\left(0,8,8,0,0,8,8,8\right).
\end{align*}
\begin{thm}\label{thm:9}
Any $\PSU\left(3\right)$-structure $\left(M^8,g,\rho\right)$ of type $\W_1\op\ldots\op\W_5$ that satisfies
\beqn
\nabla^c\T^c=0,\quad \hol\left(\nabla^c\right)=\so\left(3\right)
\eeqn
is either of type $\W_3$ and $\Scal^g>0$ or of type $\W_5$ and $\Scal^g<0$.
\end{thm}\noindent
Following \cite{Nom54}, let $\g$ be the direct sum of $\so(3)$ and $\R^8$ and define
\beqn
\left[A+X,B+Y\right]=\left(A\circ B-B\circ A-\mca{R}^c\left(X,Y\right)\right)+\left(A\left(Y\right)-B\left(X\right)-\T^c\left(X,Y\right)\right)
\eeqn
for all $A,B\in\so(3)$ and $X,Y\in\R^8$. The algebraic properties of $\T^c$ and $\mca{R}^c$ yield in both cases that $\g$, with the bracket $\left[\,\cdot\,,\,\cdot\,\right]$, becomes a Lie algebra. If $a\neq0$, this Lie algebra is isomorphic to
\beqn
\su(3)\op\su(2).
\eeqn
Now, let $\G$ be the connected, simply-connected Lie group with Lie algebra $\g$. Since $\SO(3)$ is closed in $\G$, the space
\beqn
M^8=\G/\SO(3)
\eeqn
is a smooth manifold and the inner product $\lan\,\cdot\,,\,\cdot\,\ran$ of $\R^8$ extends to an $\SO(3)$-invariant Riemannian metric $g$ on $M^8$. The canonical connection associated with the reductive decomposition $\g=\so(3)\op\R^8$ is an Ambrose-Singer connection with torsion tensor $\T^c$ and curvature tensor $\mca{R}^c$ at the origin.
\begin{thm}\label{thm:10}
There exist unique simply-connected and complete $\PSU(3)$-manifolds with
\beqn
\nabla^c\T^c=0,\quad \hol\left(\nabla^c\right)=\so\left(3\right)
\eeqn
of type $\W_3$ and of type $\W_5$. The respective manifolds are homogeneous spaces with isotropy group $\SO(3)$.
\end{thm}
%
%
%
%

%
%
%
\end{document}